\documentclass[reqno]{amsart}

\usepackage{mathrsfs}
\usepackage{amsfonts}
\usepackage{amssymb}
\usepackage{xcolor}
\usepackage{caption}

\usepackage{algorithm}
\usepackage{algorithmic}

\usepackage[colorlinks,citecolor = red,
linkcolor=blue,hyperindex,CJKbookmarks]{hyperref}
\usepackage{cite}
\usepackage{graphicx}
 \usepackage{multirow}

\usepackage{float}
\usepackage{subcaption}

\newtheorem{theorem}{Theorem}[section]
\newtheorem{lemma}[theorem]{Lemma}

\newtheorem{assumption}[theorem]{Assumption}

\theoremstyle{definition}
\newtheorem{definition}[theorem]{Definition}
\newtheorem{example}[theorem]{Example}

\newtheorem{remark}[theorem]{Remark}

\newcommand{\ba}{\begin{array}}
	\newcommand{\ea}{\end{array}}
\newcommand{\bea}{\begin{eqnarray}}
	\newcommand{\eea}{\end{eqnarray}}

\newcommand{\N}{\mathbb{N}}

\renewcommand{\P}{\mathbb{P}}

\newcommand{\R}{{\mathbb{R}}}

\numberwithin{equation}{section}

\allowdisplaybreaks

\begin{document}
	
	\title[Derivative-free LM algorithm via orthogonal spherical smoothing]
	{On the global complexity of a derivative-free Levenberg-Marquardt algorithm  via orthogonal spherical smoothing}

	\author{Xi Chen}
	\address{School of Mathematical Sciences,   Shanghai Jiao
		Tong 		University, Shanghai 200240,		China}
	\email{chenxi\_1998@sjtu.edu.cn}
	
	\author{Jinyan Fan*}
	\address{School of Mathematical Sciences,  and MOE-LSC, Shanghai Jiao
		Tong 		University, Shanghai 200240,		China}
	\email{jyfan@sjtu.edu.cn}
	
	\thanks{* Corresponding author}
	\thanks{The authors are supported by the Science and Technology
		Commission of Shanghai Municipality grant 22YF1412400 and the National Natural Science Foundation of China grant 12371307.}
	
	
	
	\maketitle

	\begin{abstract}
In this paper, we propose a derivative-free Levenberg-Marquardt algorithm for nonlinear least
squares problems, where the Jacobian matrices are approximated via orthogonal spherical
smoothing. It is shown that the gradient models which use the approximate Jacobian matrices
are probabilistically first-order accurate, and the high probability complexity bound of the algorithm is also  given.

	\end{abstract}

\section{Introduction}		
Nonlinear least squares problems appear frequently in engineering technology and scientific experiments. They have important applications in data fitting, parameter estimation and function
approximation.
In this paper, we consider the nonlinear least squares problem
	\begin{equation}\label{prob}
		\min_{x\in\mathbb{R}^n} \,
		f(x)=\frac{1}{2}\sum_{i=1}^{m}r_i(x)^2=\frac{1}{2}\|r(x)\|^2,
	\end{equation}
where   $r(x):\R^n\to \R^m$  is continuously differentiable but
its Jacobian is unavailable or expensive to compute, and $\|\cdot\|$ refers to
the standard Euclidian norm.
Thus, the traditional nonlinear least squares methods that use the exact Jacobian matrices are not applicable anymore.

%It is desirable to develop efficient derivative-free methods for such nonlinear least-squares problems.

 Brown et al. \cite{DFOLM} used coordinate-wise finite 	difference to approximate each element of the Jacobian and gave the derivative-free analogues of the Levenberg-Marquardt (LM) algorithm and Gauss-Newton algorithm for solving \eqref{prob}.
Some model-based derivative-free methods were also developed for \eqref{prob}.
For example, Zhang et al.\cite{ZhangHC} built quadratic interpolation models for each $r_i$ in (\ref{prob}) to  construct the approximate Jacobian matrix, and presented a framework for a class of derivative-free algorithms within a trust region;
Cartis et al. \cite{DFOGN} used linear interpolation models of all $r_i$ to build a quadratic model of the objective function, and presented a derivative-free version of the Gauss-Newton method, further  linear regression models of $r_i$ were used to improve the flexibility and robustness;
%Lindon
%	Roberts et	al. got approximated Jacobian by  building fully linear
%	interpolation models \cite{DFOGN} or linear regression models	\cite{
%		DFOLS} to
%	approximate each $r_i$. They also developed a model-based DFO method
%	\cite{Lindon} for solving least squares problems with sketching 	
%	dimensionality reduction technique to avoid the construction of a full
%	local models.	
Recently, Zhao et al. \cite{zhao} used  random points generated by the standard Gaussian distribution  to construct probabilistically first-order accurate Jacobian models, and showed that the LM method based on such Jacobian models converges to a first-order stationary point of the objective function with probability one.
	
In recent years, zeroth-order optimization methods that approximate gradients or stochastic gradients via randomized smoothing technique have attracted increasing attention due to
their success in solving signal processing and machine learning/deep learning problems.
The randomized smoothing approach defines a smoothed version of a function by adding  random directions that obeys a specific distribution and then taking the expectation.
	That is, the smoothed version of a function $g$ at $x$ is in the form of
	\begin{equation}\label{in1}
		g_s(x)=	\mathbb{E}_u[g(x+\gamma u)],
		\end{equation}
	where $u$ and $\gamma$ are the random direction and  smoothing parameter, respectively.
		The derivative of $g_s(x)$ can be estimated by the following
		 sample average approximation
		\begin{align}\label{in2}
			\nabla \tilde g(x)=\frac{\phi}{\gamma} \sum_{i=1}^{b}[(g(x+\gamma u_i)-g(x))u_i],
		\end{align}
		where $u_i (i=1,\ldots,b)$ are random directions, $b$ is the sample size,  and
		$\phi$ is		a constant 		determined 	by the distribution of
		$\{u_i\}$ (cf. \cite{liuzhong}).
It can also be taken as the gradient approximation of $g$.

In \cite{Flax}, Flaxman et al. approximated the gradient of a function via spherical smoothing on online convex optimization in the bandit setting, i.e.,
the random directions are uniformly sampled from the unit sphere.
  Their idea was used in the zeroth-order SVRG method, the ZO-ADMM  and the gradient-free mirror descent method as well \cite{s3, s2, s1}.
    In \cite{DavidKozak},  Kozak et al. proposed a zeroth-order stochastic subspace algorithm,
where the random directions are not only uniformly sampled from the sphere but also
		orthogonal to each other.	
Besides, Nesterov et al. \cite{YN} designed a kind of random gradient-free oracle by using Gaussian smoothing, which takes independent random vectors from Gaussian distribution as random directions.
This approach has been used in the zeroth-order Varag method, the zeroth-order conditional gradient  method, the zeroth-order online	alternating direction method and so on \cite{2018, ZOVarag, Liu1}.
Recently, Feng et al. \cite{wang} studied the gradient and Hessian approximations by using random orthogonal frames sampled from the Stiefel's manifold.

The global complexity bound is an important factor to measure the efficiency of an algorithm.
Cartis et al. \cite{orcomplx2012} studied the oracle complexity for finite-difference versions of adaptive regularization algorithms with cubics.
	Garmanjani et al. \cite{WCC2} and Gratton  et al. \cite{WCC} gave the
	%	same
	worse-case complexity  bounds  of their derivative-free  	trust-region
	methods respectively. Cao et al. \cite{Cao2023} derived the exponentially decaying tail
	bounds on the iteration complexity of a modified trust-region method designed for minimizing objective functions whose value and gradient and Hessian
estimates are computed with noise.

In this paper, we approximate the Jacobian matrices via the orthogonal spherical smoothing,
and propose a derivative-free analogue of the LM algorithm for nonlinear least squares problems \eqref{prob}.
We introduce a new definition of probabilistically first-order accurate gradient models
and show that the gradient models with Jacobian matrices approximated via the
orthogonal spherical smoothing are probabilistically first-order accurate if
the number of random perturbations is chosen appropriately.
%	Further, we prove that the derivative-free LM algorithm based on such
%Jacobian approximations converges globally with probability one.
%The high probability complexity bound of the algorithm is also derived.
	Further, the high probability complexity bound of the algorithm is
	derived.

 The paper is organized as follows. In Section \ref{sec2},
	we show how to construct approximate Jacobian matrices via orthogonal spherical smoothing.
	In Section 	\ref{sec3}, we propose a derivative-free LM algorithm based on the approximate Jacobian matrices and introduce a new definition of  probabilistically first-order accurate
	gradient models.
%	It is shown that the proposed LM algorithm converges globally
%almost-surely in Section \ref{sec4}.
 The complexity analysis of the algorithm is given in Section	\ref{sec5}.
	Finally, some numerical experiments are presented in Section \ref{sec6}.

%we conclude the paper in Section \ref{sec6}.

	\section{Approximating Jacobian matrices via orthogonal spherical smoothing}\label{sec2}
	In this section, we show how to approximate the Jacobian matrices via orthogonal spherical smoothing, and give a bound of the variance of approximate gradients.
	
Denote by $\Tilde{J}(x)$ the approximation of the Jacobian matrix $J(x)$, i.e.,
	\begin{align}\label{JA}
		\Tilde{J}(x)= \left(\nabla \Tilde{r}_1(x), \ldots,
		\nabla \Tilde{r}_m(x)\right)^T,
			\end{align}
	where $\nabla \Tilde{r}_i(x)$ is the approximation of 	$\nabla r_i(x)$.
Define the spherical smoothed version of $r_i(x)$ as
	\begin{equation}\label{spsmoo}
		r_{s,i}(x)=\mathbb{E}_{v\sim U(\mathbb B)} [r_i(x+\gamma v)],
	\end{equation}
	where $\mathbb B$ represents the unit ball in $\R^n$ and  $U(\mathbb B)$
	the uniform distribution on  $\mathbb B$.
It follows from \cite{Flax} that
	\begin{equation}\label{ssmooder}
		\nabla r_{s,i} (x)= \frac{n}{\gamma} \mathbb{E}_{u\sim U(\mathbb S)}
		[r_i(x+\gamma u)u],
	\end{equation}
	where $\mathbb S$ represents the unit sphere in $\R^n$ and $U(\mathbb S)$ the
	uniform distribution on $\mathbb S$. 	

		We take
	\begin{align}\label{nabr}
		\nabla \Tilde{r}_i(x)= \frac{n}{b} \sum_{j=1}^b \frac{r_i(x+\gamma
			u_j)-r_i(x)}{\gamma }u_j,
	\end{align}
	where $[u_1,\ldots,u_b]$ is  uniformly sampled from the
	Stiefel
	manifold $ St(n, b)=\{U\in \R^{n \times b}: U^TU=I_b\}$ with
	$1\leq b  \leq n$. (\ref{nabr}) can be regarded as a
	special case of  gradient approximation using  spherical 	smoothing
	described in \cite{Flax} and can be
	obtained by using Gram-Schmidt orthogonalization on $b$ independent random
	vectors with identical distribution
	$\mathcal{N}(0,I)$ (cf. \cite{wang,Chikuse2003}).
	By 	\cite{DavidKozak}, the  distribution for any such
	$u_j$ is
	uniform over
	$ \mathbb S   $.
	Thus,
	\begin{align}
\mathbb{E}_u [\nabla
	\Tilde{r}_i(x)] =\frac{n}{b} \sum_{j=1}^b \mathbb{E}_u
	\left[\frac{r_i(x+\gamma
		u_j)-r_i(x)}{\gamma }u_j\right]	=\frac{1}{b} \sum_{j=1}^b \nabla
	r_{s,i}(x)=\nabla  r_{s,i}(x).
\end{align}

	Denote the stochastic approximation of the gradient of the objective function $\nabla f(x)$ by $\nabla
	\tilde{f} (x)
	=\tilde{J}(x)^Tr(x)$.
In the following, we estimate the error between $\nabla f(x)$ and  $\nabla \Tilde{f}(x)$.	
Note that
	\begin{equation}\label{2.1.1}
		\left\|\nabla \tilde{f}(x)-\nabla f(x)\right\| \leq \left\|\nabla
		\tilde{f}(x)-\mathbb{E}_u[\nabla\tilde{f}(x)]\right\| +\left\|\nabla
		f(x)- \mathbb{E}_u[\nabla\tilde{f}(x)] \right\|.
	\end{equation}
	The first term on the right side of (\ref{2.1.1})  is
	related to the variance of approximate gradient $\nabla \tilde{f}(x)$,
while the second term is the difference  between the gradient and the expectation of approximate gradient.
	The bounds for these two terms will be derived respectively.

We make the following   assumptions.
Without loss of generality, we assume $\|r_0\|=\|r(x_0)\| \ne 0$ throughout the paper.

	\begin{assumption}\label{ass:2} $J(x)$ is Lipschitz continuous, i.e.,
		there exists $\kappa_{lj}>0 $ such that
		\begin{equation}\label{lipj}
			\|J(x)-J(y)\| \leq \kappa_{lj} \|x-y\|,\quad \forall x,y \in \mathbb{R}^n.
		\end{equation}
	\end{assumption}
\eqref{lipj} implies that there exist $\kappa_i>0 (i=1,\ldots,m) $ such that
\begin{align}\label{li}
		\|\nabla r_i(x)- \nabla r_i(y)\| \leq \kappa_i \|x-y\|, \quad \forall x,y \in \mathbb{R}^n.
\end{align}
Let  $\kappa_{\max}=\displaystyle\max_{i=1,\ldots,m} \kappa_{i}$.

	\begin{lemma}\label{lem2.6}
	Suppose that Assumption \ref{ass:2} holds. Then
		\begin{align}\label{lem2eq}
			\|\mathbb{E}_u [\nabla \tilde{f}(x)]-\nabla f(x)\| \leq
			\frac{\sqrt{m}n	\kappa_{\max} \gamma}{n+1} \|r(x)\|.
		\end{align}
		
	\end{lemma}
	\begin{proof}
		It follows from \cite{wang} that
\begin{equation}\label{lem1eq}
			\|\mathbb{E}_u [\nabla \tilde{r}_i(x)] - \nabla r_i(x)\| \leq
			\frac{n
				\kappa_{i}\gamma }{n+1}.
		\end{equation}
Hence,
		\begin{align*}
			\|\mathbb{E}_u[\nabla \tilde{f}(x)]-\nabla f(x)\| &=
			\|\mathbb{E}_u[\tilde{J}(x)^Tr(x)]-J^T(x) r(x)\|\\
			& =\left\|\mathbb{E}_u
			\left[\sum_{i=1}^m
			\nabla
			\tilde{r}_i(x) r_i(x)\right]-\sum_{i=1}^m \nabla r_i(x) r_i(x)\right\| \\
			&=\left\|\sum_{i=1}^m\left( \mathbb{E}_u[\nabla  \tilde{r}_i(x)] -
			\nabla  r_i(x)\right) r_i(x) \right\| \\
			& \leq \sum_{i=1}^m \|
			\mathbb{E}_u[\nabla  \tilde{r}_i(x)] -  \nabla  r_i(x)\| |r_i(x) |\\
&\leq			\sum_{i=1}^m
			\frac{n \kappa_{i}\gamma}{n+1}|r_i(x)|\\
& = \frac{n \kappa_{\max} \gamma}{n+1} \|r(x)\|_1.
		\end{align*}
This, together with
 \begin{align}\label{norm1-2}
 \|r(x)\|_1 \leq \sqrt{m}\|r(x)\|,
 \end{align}
	yields   	 \eqref{lem2eq}.
	\end{proof}

	%Next, we derive a bound for the variance of the approximate gradient.

	\begin{lemma}\label{lem31}
	Suppose that Assumption \ref{ass:2} holds. Then
		\begin{align}\label{lem3eq}
			&\mathbb{E}_u \left[ \left\|\nabla \tilde{f}(x)-\mathbb{E}_u
			[\nabla
			\tilde{f}(x)] 	\right\|^2\right] \nonumber\\
			&\leq \left(\frac{n}{b}- 1\right)
			\|\nabla f(x)\|^2 + \left(\frac{n}{b}- 1\right) \sqrt{m}n\kappa_{\max}\gamma
			\|r(x)\| 	\|\nabla
			f(x)\| + \frac{m n^2 \kappa_{\max}^2 \gamma^2}{4b}\|r(x)\|^2.
		\end{align}
	\end{lemma}
	\begin{proof}
It holds that
		\begin{align*}
			\mathbb{E}_u\left[ \left\|\nabla \tilde{f}(x)-\mathbb{E}_u [\nabla
			\tilde{f}(x)] 	\right\|^2\right] =\mathbb{E}_u \left[ \| \nabla
			\tilde{f}(x)
			\|^2\right] -\|\mathbb{E}_u [\nabla \tilde{f}(x)] \|^2.
		\end{align*}
By the Taylor expansion, for $i=1,\ldots,m$,
\begin{align*}
\frac{r_i(x+\gamma u_j)-r_i(x)}{\gamma} = \nabla
		r_i(x)^T
		u_j+ \frac{\gamma }{2} u_j^T r''_i(\xi_i)u_j,
\end{align*}
where $\xi_i\in [x,x+\gamma u_j]$.
 Denote
\begin{align*}
 \phi_j=u_j^T \left(
		\sum_{i=1}^m r''_i(\xi_i) r_i(x)\right) u_j, \quad j=1,\ldots, b.
\end{align*}
	Then  by \eqref{nabr},
\begin{align*}
	&\mathbb{E}_u \left[ \| \nabla \tilde{f}(x)
	\|^2\right] =  \mathbb{E}_u \left[ \| \tilde{J}(x)^T r(x)
	\|^2\right]
	=  \mathbb{E}_u \left[ \left\| \sum_{i=1}^m \nabla \tilde{r}_i(x)
	r_i(x)	
	\right\|^2\right]  \nonumber\\
	&=\mathbb{E}_u \left[\left\|  \sum_{i=1}^m \left( \frac{n}{b}
	\sum_{j=1}^b
	\frac{r_i(x+\gamma u_j)-r_i(x)}{\gamma} u_j  \right) r_i(x)
	\right\|^2\right]
	\nonumber\\
	&= \frac{n^2}{b^2} \mathbb{E}_u \left[\left\|  \sum_{i=1}^m
	\sum_{j=1}^b
	\left( 	 \nabla r_i(x)^T
	u_j+ \frac{\gamma }{2} u_j^T r''_i(\xi_i)u_j  \right)u_j r_i(x)
	\right\|^2\right]
	\nonumber\\
	&=\frac{n^2}{b^2} \mathbb{E}_u \left[\left\| \sum_{j=1}^b
	u_j^T\left(
	\sum_{i=1}^m  \nabla r_i(x)  r_i(x)\right)
	u_j+ \frac{\gamma}{2}  \sum_{j=1}^b u_j^T  \left(
	\sum_{i=1}^m r''_i(\xi_i) r_i(x)\right)  u_j u_j
	\right\|^2\right]
	\nonumber\\
	&=\frac{n^2}{b^2} \mathbb{E}_u \left[\left\| \sum_{j=1}^b
	\left(u_j^T
	\nabla f(x)\right) u_j+ \frac{\gamma}{2}  \sum_{j=1}^b \phi_j u_j
	\right\|^2\right] \nonumber\\
	&= \frac{n^2}{b^2} \mathbb{E}_u \left[\left\| \sum_{j=1}^b
	\left(u_j^T
	\nabla f(x)\right) u_j\right\|^2 \right] +
	\frac{n^2}{b^2}\mathbb{E}_u
	\left[ \gamma \left(
	\sum_{j=1}^b \left(u_j^T
	\nabla f(x)\right) u_j^T\right)\left(\sum_{j=1}^b \phi_j u_j
	\right)\right]   \nonumber\\
	& \quad +\frac{n^2}{b^2}\mathbb{E}_u \left[
	\frac{\gamma^2}{4} \left(\sum_{j=1}^b \phi_j u_j^T  \right)
	\left(\sum_{j=1}^b
	\phi_j u_j  \right)  \right].
\end{align*}

 By \eqref{li} and \eqref{norm1-2},
		\begin{equation}\label{R}
			|\phi_j| \leq \|u_j\|^2
			\left\|\sum_{i=1}^m r''_i(\xi_i) r_i(x)\right\| \leq \sum_{i=1}^m
			\|r''_i(\xi_i)\| |
			r_i(x)| \leq \sqrt{m} \kappa_{\max} \|r(x)\|.
		\end{equation}
It follows from \cite{wang} that
		\begin{equation}\label{uu}
			\mathbb{E}_u [u_j u_j^T]=\frac{1}{n}I.
		\end{equation}
%		(cf. \cite{wang}) that
		Since $u_j^Tu_k=0$ for all $j\not=k, j,k=1,\ldots,b$, by \eqref{R} and \eqref{uu},
		\begin{align}\label{tem1}
			&\mathbb{E}_u \left[ \| \nabla \tilde{f}(x)
			\|^2\right] \nonumber\\
			& =\frac{n^2}{b^2}\sum_{j=1}^b\mathbb{E}_u \left[ \nabla
			f(x)^T u_j u_j^T \nabla f(x) \right] + \frac{n^2\gamma}{b^2}
			\sum_{j=1}^b  \mathbb{E}_u \left[ \nabla 	f(x)^T u_j \phi_j
			\right]  +
			\frac{n^2\gamma^2}{4b^2}  \sum_{j=1}^b \mathbb{E}_u \left[ \phi_j^2
			\right]
			\nonumber\\
			%&=  \frac{n}{b}\|\nabla f(x)\|^2 +
%			\frac{n^2\gamma}{b^2}
%			\sum_{j=1}^b  \mathbb{E}_u \left[ \nabla 	f(x)^T u_j \phi_j \right]
%			+ \frac{n^2\gamma^2}{4b^2}  \sum_{j=1}^b \mathbb{E}_u \left[ \phi_j^2
%			\right]
%			\nonumber\\
			& \leq  \frac{n}{b}\|\nabla f(x)\|^2 +
			\frac{n^2\gamma}{b^2}
			\sum_{j=1}^b  \mathbb{E}_u \left[ \nabla 	f(x)^T u_j \phi_j \right]
			+ \frac{m n^2 \kappa_{\max}^2\gamma^2}{4b}\|r(x)\|^2.
		\end{align}
		
		Now we consider $\|\mathbb{E}_u [\nabla \tilde{f}(x) ] \|^2$. By \eqref{uu},
		\begin{align}\label{tem2}
			&	\|\mathbb{E}_u [\nabla \tilde{f}(x)] \|^2 = 	\left\|\mathbb{E}_u 	\left[ \sum_{i=1}^m
			\nabla \tilde{r}_i(x) r_i(x)\right] \right\|^2\nonumber \\
			& = \left\| \mathbb{E}_u \left[
			\frac{n}{b}
			\sum_{i=1}^m \left(  \sum_{j=1}^b
			\frac{r_i(x+\gamma u_j)-r_i(x)}{\gamma} u_j  \right) r_i(x)  \right]
			\right\|^2  \nonumber \\
			& = \frac{n^2}{b^2} \left\| \mathbb{E}_u \left[ \sum_{i=1}^m
			\sum_{j=1}^b
			\left( 	 \nabla r_i(x)^T
			u_j+ \frac{\gamma}{2}  u_j^T r''_i(\xi_i)u_j  \right)u_j r_i(x)
			\right] \right\|^2
			\nonumber\\
			&=\frac{n^2}{b^2} \left\|\mathbb{E}_u \left[ \sum_{j=1}^b
			u_j^T\left(
			\sum_{i=1}^m  \nabla r_i(x)  r_i(x)\right)
			u_j+ \frac{\gamma}{2}  \sum_{j=1}^b u_j^T  \left(
			\sum_{i=1}^m r''_i(\xi_i) r_i(x)\right)  u_j u_j
			\right]\right\|^2
			\nonumber\\
			&=\frac{n^2}{b^2}\left\| \mathbb{E}_u \left[ \sum_{j=1}^b
			\left(u_j^T
			\nabla f(x)\right) u_j\right] + \frac{\gamma}{2}
			\sum_{j=1}^b\mathbb{E}_u
			\left[ \phi_j u_j \right] \right\|^2\nonumber\\
			&=\frac{n^2}{b^2} \left\|\sum_{j=1}^b \mathbb{E}_u \left[
			u_ju_j^T
			\nabla f(x) \right] \right\|^2 + \frac{n^2\gamma}{b^2}  \left(
			\sum_{j=1}^b
			\mathbb{E}_u \left[  u_ju_j^T \nabla f(x) \right] \right)^T
			\left(
			\sum_{j=1}^b\mathbb{E}_u 	\left[ \phi_j u_j \right] \right)
			\nonumber\\
			& \quad + \frac{n^2 \gamma^2}{4b^2}  \left(
			\sum_{j=1}^b\mathbb{E}_u 	\left[ \phi_j u_j \right]
			\right)^T\left(
			\sum_{j=1}^b\mathbb{E}_u 	\left[ \phi_j u_j \right]
			\right)\nonumber\\
			& = \frac{n^2}{b^2}\left\| \frac{b}{n} \nabla
			f(x)\right\|^2 +
			\frac{n\gamma}{b} \nabla f(x)^T \left(
			\sum_{j=1}^b\mathbb{E}_u 	\left[ \phi_j u_j \right] \right)
			\nonumber\\
			& \qquad +
			\frac{n^2
				\gamma^2}{4b^2}  \left(
			\sum_{j=1}^b\mathbb{E}_u 	\left[ \phi_j u_j \right]
			\right)^T\left(
			\sum_{j=1}^b\mathbb{E}_u 	\left[ \phi_j u_j \right] \right)
			\nonumber\\
			& = \|\nabla f(x)\|^2+ \frac{n\gamma}{b}
			\sum_{j=1}^b  \mathbb{E}_u \left[ \nabla 	f(x)^T u_j \phi_j
			\right]+
			\frac{n^2
				\gamma^2}{4b^2}  \left(
			\sum_{j=1}^b\mathbb{E}_u 	\left[ \phi_j u_j \right]
			\right)^T\left(
			\sum_{j=1}^b\mathbb{E}_u 	\left[ \phi_j u_j \right] \right).
		\end{align}
		
		Combining \eqref{R}, (\ref{tem1}) and (\ref{tem2}), we obtain
		\begin{align*}
			&	\mathbb{E}_u\left[ \left\|\nabla \tilde{f}(x)-\mathbb{E}_u
			[\nabla
			\tilde{f}(x)] 	\right\|^2\right] =\mathbb{E}_u \left[ \| \nabla
			\tilde{f}(x)
			\|^2\right] -\|\mathbb{E}_u [\nabla \tilde{f}(x)] \|^2 \\
			%&	=\left(\frac{n}{b}- 1\right) \|\nabla f(x)\|^2 +
%			\frac{n}{b}\left(\frac{n}{b}- 1 \right) \gamma
%			\sum_{j=1}^b  \mathbb{E}_u \left[ \nabla 	f(x)^T u_j \phi_j \right]
%			+ \frac{n^2 \kappa_{\max}^2 m \|r(x)\|^2}{4b} \gamma^2 \\
%			& \quad -\frac{n^2
%				\gamma^2}{4b^2}  \left(
%			\sum_{j=1}^b\mathbb{E}_u 	\left[ \phi_j u_j \right]
%			\right)^T\left(
%			\sum_{j=1}^b\mathbb{E}_u 	\left[ \phi_j u_j \right] \right) \\
			& \leq \left(\frac{n}{b}- 1\right) \|\nabla f(x)\|^2 +
			\frac{n}{b}\left(\frac{n}{b}- 1 \right) \gamma
			\sum_{j=1}^b  \mathbb{E}_u \left[ \nabla 	f(x)^T u_j \phi_j \right]
			+ \frac{m n^2 \kappa_{\max}^2 \gamma^2}{4b}\|r(x)\|^2\\
			& \leq \left(\frac{n}{b}- 1\right) \|\nabla f(x)\|^2 +
			\frac{n}{b}\left(\frac{n}{b}- 1 \right) \gamma \|\nabla f(x)\|	
			\sum_{j=1}^b  \mathbb{E}_u[\|u_j\| |\phi_j|]+ \frac{m n^2 \kappa_{\max}^2
				\gamma^2}{4b}\|r(x)\|^2 \\
			& \leq     \left(\frac{n}{b}- 1\right)
			\|\nabla
			f(x)\|^2 + \left(\frac{n}{b}- 1\right) \sqrt{m}n\kappa_{\max}\gamma\|\nabla
			f(x)\| \|r(x)\|
			+ \frac{m n^2 \kappa_{\max}^2\gamma^2}{4b}\|r(x)\|^2.
		\end{align*}
	\end{proof}

		\begin{lemma}\label{lem2.1.1}
			Suppose that Assumptions \ref{ass:2} holds.
 If $b \ge \frac{n}{2}$, then for any $s>0$,
			\begin{equation}\label{lem2.1.1eq}
				\mathbb{P}\left( \left\|\nabla
				\tilde{f}(x)-\mathbb{E}_u[\nabla\tilde{f}(x)]\right\| \leq s
				\right)
				\ge 1- \frac{\left(\left(\frac{n}{b}- 1\right)^{1/2}
					\|\nabla f(x)\|+ \frac{\sqrt{m}n\kappa_{\max}\gamma}{2} \|r(x)\|
					\right)^2}{s^2}.
			\end{equation}

		\end{lemma}
		
		\begin{proof}
			Since $1\leq b \leq n$ and $b \ge \frac{n}{2}$, we have $0\leq \left(\frac{n}{b}- 1\right) \leq			1$.
			Define $p(x):=\sqrt{m}n\kappa_{\max}\|r(x)\|$. By Lemma \ref{lem31},
			\begin{align}\label{2.1.2}
				&\mathbb{E}_u\left[ \left\|\nabla \tilde{f}(x)-\mathbb{E}_u
				[\nabla
				\tilde{f}(x)] 	\right\|^2\right] \nonumber\\
				&\leq \left(\frac{n}{b}- 1\right)
				\|\nabla f(x)\|^2 + \left(\frac{n}{b}- 1\right)\gamma p(x)	
				\|\nabla
				f(x)\| + \frac{\gamma^2}{4b}p(x)^2\nonumber\\
				&\leq \left(\frac{n}{b}- 1\right)\|\nabla f(x)\|^2 +
				\left(\frac{n}{b}- 1\right)^{1/2}
			\gamma p(x)	\|\nabla f(x)\|+ \frac{\gamma^2}{4}p(x)^2
				\nonumber\\
				&= \left(\left(\frac{n}{b}- 1\right)^{1/2}\|\nabla f(x)\|+
				\frac{\gamma}{2} p(x) \right)^2.
			\end{align}
			Then by  the  multivariate Chebyshev's inequality \cite{Chebyshevproof},
			for any $s>0$,
			\begin{align}\label{mchby}
				\mathbb P( \|\nabla \tilde{f}(x)-\mathbb{E}_u [\nabla
				\tilde{f}(x)] \| \leq s) &\geq
				1-\frac{	\mathbb{E}_u\left[ \left\|\nabla
					\tilde{f}(x)-\mathbb{E}_u
					[\nabla \tilde{f}(x)] 	\right\|^2\right]}{s^2}\nonumber \\
&\ge 1- \frac{\left(\left(\frac{n}{b}- 1\right)^{1/2}\|\nabla
					f(x)\|+
					\frac{\gamma}{2} p(x)\right)^2 }{s^2}.
			\end{align}
					This gives (\ref{lem2.1.1eq}).
		\end{proof}

		\section{ Derivative-free LM algorithm and probabilistic gradient
			models }\label{sec3}
		
In this section, we first propose a derivative-free LM algorithm for nonlinear
least squares problems, where the approximate Jacobian matrices are
constructed via the orthogonal spherical smoothing as in   Section
\ref{sec2}.  Then we introduce a new definition of probabilistic gradient
			models and show that the random gradient models generated by Algorithm \ref{alg1} are probabilistically first-order accurate if the smoothing parameter and sampling size  are taken appropriately.

As described in the previous section, the Jacobian models and the gradient
models are stochastic because the direction vectors in smoothing are random.
%  Consider random models of the form $M_k$. /
  	Denote $M_k$ as
the  random model   encompassing all
	randomness of the $k$-th ($k\ge 1$) iteration.
Denote by $J_{M_k}$ and $\nabla	f_{M_k}$ the random Jacobian models and random gradient models, respectively.
We use the notation $m_k=M_k(w_k), J_{m_k}=J_{M_k}(w_k)$
and $\nabla f_{m_k}=\nabla	f_{M_k}(w_k)$ for their
realizations. Consequently,   the
 randomness of the models implies the randomness of
$x_k=X_k(w_k)$, $d_k=D_k(w_k)$, $\lambda_k=\Lambda_k(w_k)$ and $\theta_k=\Theta_k(w_k)$, generated by the corresponding optimization algorithm.

	At the $k$-th iteration, the LM method solves the linear equations
		\begin{equation}\label{solve}
					( J_{m_k}^TJ_{m_k}+\lambda_k I)d
					= -J_{m_k}^T r(x_k)\quad \mbox{with}\quad
					\lambda_k=\theta_k \|J_{m_k}^T r(x_k)\|
					%\|r(x_k)\|
				\end{equation}
to obtain the step $d_k$,
where $J_{m_k}$	is the approximate Jacobian matrix at $x_k$, and $\lambda_k$ is the LM parameter. 	
Define the ratio of the actual reduction to the predicted reduction of the objective function as
\begin{equation}\label{ratio}
			\rho_k=\frac{Ared_k}{Pred_k}=
				\frac{\|r(x_k)\|^2-\|r(x_k+d_k)\|^2}{\|r(x_k)\|^2-\|r(x_k)+
					J_{m_k}
					d_k\|^2 }.
		\end{equation}
It plays a key role in deciding whether the step $d_k$ is successful and how to update the parameter $\theta_k$. If $\rho_k$ is smaller than a small positive constant, we reject $d_k$ and increase $\theta_k$. Otherwise we accept $d_k$ and update $\theta_k$ based on its relationship with the gradient approximation norm.

		\begin{algorithm}%\label{ALG1}
			\renewcommand{\algorithmicrequire}{\textbf{Input:}}
			\renewcommand{\algorithmicensure}{\textbf{Output:}}
			\caption{ Derivative-free LM algorithm via orthogonal
			spherical 				smoothing}
			\label{alg1}
			\begin{algorithmic}[1] %用数字每行标号
				\REQUIRE Initial $x_0$, $b>0$, $a_1>1>a_2>0
				 (a_1a_2< 1)$,
				$0<p_0<1$, $p_0<p_1<p_2$, $\theta_1\ge \theta_{\min}>0$,
							$\epsilon_0\geq 0$, $k:=0$.
			%	\FOR{$k=0,1,\cdots,k_{\max}$}
				\STATE  Construct approximate Jacobian $J_{m_k}$ by (\ref{JA}) and (\ref{nabr}).
				\STATE  If $\|J_{m_k}^T r(x_k)\|\leq \epsilon_0$, stop; else
				solve \eqref{solve}	to obtain  $d_k$.
				\STATE Compute 	$\rho_k$ by \eqref{ratio}.
				\STATE Set
				\begin{align}
x_{k+1}=
				\begin{cases}
					x_k+d_k,& \mbox{if}\ \rho_k\ge p_0,\\
					x_k, & \mbox{otherwise}.
				\end{cases}
\end{align}
				If $\rho_k< p_0$, compute $\theta_{k+1}=a_1\theta_k$; otherwise compute
				\begin{equation}\label{2.7a}
					\theta_{k+1}=
					\begin{cases}
						a_1\theta_k,  & \mbox{if}\  \|J_{m_k}^T r(x_k)\|<
						\frac{p_1}{\theta_k},\\
						\theta_k,   & \mbox{if}\
						\|J_{m_k}^T r(x_k)\|\in
										[\frac{p_1}{\theta_k},\frac{p_2}{\theta_k})		
						,\\
						\max\{a_2\theta_k,\theta_{\min}\}, & \mbox{otherwise}.
					\end{cases}
				\end{equation}
				%\STATE $\omega^{k+1}=\omega^k+d^k$.
				\STATE Set $k:=k+1$;  Go to
				step 1.
				
		%		\ENDFOR
				\ENSURE $x$
			\end{algorithmic}
		\end{algorithm}

		\begin{definition}\label{def2.1}
			Given constants $\alpha\in (0,1)$,
			$\xi_1 \ge 0$ and $\xi_2 \ge 0$. The sequence of
			random gradient models $\{\nabla f_{M_k}(X_k)\}$ is said to be
			$\alpha$-probabilistically $(\xi_1,\xi_2)$-first-order
			accurate, for corresponding
			sequences $\{X_k\}$, if
			the events
			\begin{equation}\label{AK}
				A_k=\left\lbrace \|\nabla f_{M_k}(X_k)-\nabla f(X_k)\|\leq
				\xi_1
				\| \nabla f(X_k)\|+\xi_2 \| D_{k-1}\|  \right\rbrace
			\end{equation}
			satisfy the  condition
			\begin{equation}\label{relerr}
					\mathbb{P}\left(A_k\mid
					\mathcal{F}_{k-1}\right) \ge
				\alpha,
			\end{equation}
			where $\mathcal{F}_{k-1}$ is the $\sigma$-algebra generated by all
			randomness before the $k$-th ($k\ge 1$) iteration.
			%	after arriving at the $k$-th iteration.
			%$X_k$, but before calculating $\tilde{J}_k$.	
	
		Correspondingly,
			$\nabla f_{m_k}(x_k)$ is said to
			be $(\xi_1,\xi_2)$-first-order accurate if
%			 \textcolor{blue}{event $A_k$ happens.}
			\begin{equation}\label{2.2realiz}
				\|\nabla f_{m_k}(x_k)-\nabla f(x_k)\|\leq \xi_1\|\nabla
				f(x_k)\| +\xi_2 \|d_{k-1}\|.
			\end{equation}
%		\textcolor{blue}{ holds with probability 1.}
		\end{definition}

		The following theorem indicates that  the random gradient models
		$\{\nabla
		f_{M_k}(X_k)\}$ constructed by Algorithm \ref{alg1} are  	
		probabilistically first-order accurate if the smoothing parameter  and sampling
		size are taken appropriately.

		\begin{theorem}\label{th2.1.2}
			Suppose that Assumptions \ref{ass:2} holds.
%			  for every realization
%			of			Algorithm 			\ref{alg1}.
%			and consider the $k$-th iteration of that realization.
For  $\alpha \in (0,1)$, if $b \ge
			\frac{n}{2}$ and 	$\gamma_k=\|d_{k-1}\|$,
			then the gradient models $\{\nabla f_{M_k}(X_k)\}$  generated by 	
			Algorithm \ref{alg1}  are
			$\alpha$-probabilistically $(\xi_1,\xi_2)$-first-order accurate with
			$\xi_1=(\frac{\frac{n}{b}-1}{1-\alpha})^{1/2}$ and
			$\xi_2=\sqrt{m}\kappa_{\max}\left( \frac{n}{2(1-\alpha)^{1/2}}
			+1	\right)\|r_0\|$.
		\end{theorem}
		\begin{proof}
%			\textcolor{blue}{
%	First,	consider the $k$-th realization $x_k$.
%		}
Denote  $\bar{s}_k^2:=\left(
			\left(\frac{n}{b}- 1\right)^{1/2}	\|\nabla f(x_k)\|+ \frac{\sqrt{m}n\kappa_{\max}\|r_0\|}{2}\gamma_k
			\right)^2 $.
		Since $\|r(x_k)\| \leq \|r_0\|$, by Lemma \ref{lem2.1.1},
			\begin{equation}\label{th2term1}
				\mathbb{P}\left( \left\|\nabla
				f_{M_k} (x_k)-\mathbb{E}_u[\nabla f_{M_k}(x_k)]\right\| \leq
				s_k
				\right)
				\ge 1- \frac{\bar{s}_k^2 }{s_k^2}.
			\end{equation}
			Take $s_k=\bar{s}_k (1-\alpha)^{-1/2}$. Then (\ref{th2term1}) gives
			\begin{equation}\label{th2term2}
				\mathbb{P}\left( \left\|\nabla
				 f_{M_k} (x_k)-\mathbb{E}_u[\nabla f_{M_k} (x_k)]\right\| \leq
				s_k
				\right)
				\ge \alpha.
			\end{equation}
By Lemma \ref{lem2.6},
\begin{align}\label{err1}
			\left\|\nabla
			f(x_k)- \mathbb{E}_u[\nabla f_{M_k}(x_k)] \right\|\leq
			\frac{\sqrt{m}n	\kappa_{\max} \gamma_k}{n+1} \|r(x_k)\| \leq
			\sqrt{m}\kappa_{\max}\gamma_k\|r(x_k)\|.
		\end{align}
This, together with (\ref{2.1.1}) and (\ref{th2term2}), yields
			\begin{align*}
\mathbb{P}\Bigg( \left\|\nabla
				 f_{M_k} (x_k)-\nabla f(x_k)\right\| \leq&
				\left(\frac{\frac{n}{b}-1}{1-\alpha}\right)^{1/2}\|\nabla f(x_k)\|   \\
&  \quad
  +\frac{\sqrt{m}n\kappa_{\max}\|r_0\|\gamma_k}{2(1-\alpha)^{1/2}}  +\sqrt{m}\kappa_{\max}\|r_0\|\gamma_k
				\Bigg)  \ge \alpha.
			\end{align*}
	Note that  $\gamma_k=\|d_{k-1}\|$, we have
			\begin{align*}
			\mathbb{P}\Bigg( \|\nabla
			f_{M_k}(X_k)-&\nabla f(X_k)\| \leq
			\left(\frac{\frac{n}{b}-1}{1-\alpha}\right)^{1/2}\|\nabla
			f(X_k)\|   \\
			&
			+ \sqrt{m}\kappa_{\max}\left( \frac{n}{2(1-\alpha)^{1/2}}
			+1	\right)\|r_0\| \|D_{k-1}\| \mid \mathcal{F}_{k-1}
			\Bigg)  \ge \alpha.
		\end{align*}					
 By 	Definition \ref{def2.1}, $\{\nabla f_{M_k}(X_k)$\}  are
				$\alpha$-probabilistically $(\xi_1,\xi_2)$-first-order
				accurate
	with
				$\xi_1=(\frac{\frac{n}{b}-1}{1-\alpha})^{1/2}$ and
				$\xi_2=\sqrt{m}\kappa_{\max}\left( \frac{n}{2(1-\alpha)^{1/2}}
				+1	\right)\|r_0\|$.
	%			by Definition \ref{def2.1},
%$\nabla\tilde{f}(x_k)$  is
%			$\alpha$-probabilistically $(\xi_1,\xi_2)$-first-order accurate
%with
%			$\xi_1=(\frac{\frac{n}{b}-1}{1-\alpha})^{1/2}$ and
%			$\xi_2=\sqrt{m}\kappa_{\max}\left( \frac{n}{2(1-\alpha)^{1/2}}
%			+1	\right)\|r_0\|$.
 		\end{proof}

		\begin{lemma}\label{lem3.4}
			Suppose that Assumption \ref{ass:2} holds.
%				for every		realization  of Algorithm \ref{alg1}.
			% and consider the $k$-th  %iteration of that realization.
For  constants $ \eta \in (0,1)$ and
			$\alpha \in (0,1)$,		if
			$\gamma_k=\|d_{k-1}\|$, 			$b \ge
			\max\left\{n\Big/\left(
			\frac{\eta^2(1-p_0)^2(1-\alpha)}{\left( 2+\eta
				(1-p_0)\right) ^2} +1\right), \frac{n}{2}\right\}	$,
			then the gradient models
			$\{\nabla f_{M_k}(X_k)\}$  generated
			by 	
			Algorithm \ref{alg1}  are 			
			$\alpha$-probabilistically $(\xi_1,\xi_2)$-first-order accurate
			with
			$\xi_1 \leq  \frac{\eta (1-p_0)}{2+ \eta (1-p_0)}$.

		\end{lemma}	
	\begin{proof}
			It follows from Theorem \ref{th2.1.2} that
		$\{\nabla f_{M_k}(X_k)\}$ are	
		$\alpha$-probabilistically $(\xi_1,\xi_2)$-first-order accurate
		with
		$\xi_1=(\frac{\frac{n}{b}-1}{1-\alpha})^{1/2}$ if
		$\gamma_k=\|d_{k-1}\|$ and
		$b \ge  \frac{n}{2}$.
		
		Since $b \ge
		\max\left\{n \Big/\left(  \frac{\eta^2(1-p_0)^2(1-\alpha)}{\left( 2+\eta
			(1-p_0)\right) ^2} +1 \right), \frac{n}{2}\right\}	$ and $1-\alpha
		> 0$, we have
		\begin{align*}
			\xi_1=\left(\frac{\frac{n}{b}-1}{1-\alpha}\right)^{1/2}
			\leq \frac{\left(\frac{\eta^2(1-p_0)^2(1-\alpha)}{\left( 2+\eta
					(1-p_0)\right) ^2}\right)^{1/2}}{(1-\alpha)^{1/2}}=
			\frac{\eta (1-p_0)}{2+ \eta (1-p_0)}.
		\end{align*}
			\end{proof}
		
\begin{remark}\label{remark}
		By Lemma \ref{lem3.4}, if we take $\gamma_k=\|d_{k-1}\|$ and $b \ge
		n\Big/\left(\frac{\eta^2(1-p_0)^2}{2 \left( 2+\eta
		(1-p_0)\right) ^2} +1\right) $,
then	$\{\nabla
	f_{M_k}(X_k)\}$ are	
	$\frac{1}{2}$-probabilistically
	$(\xi_1,\xi_2)$-first-order 	accurate with
	$\xi_1 \leq  \frac{\eta (1-p_0)}{2+ \eta (1-p_0)}$.
\end{remark}

Theorem \ref{th2.1.2} and Lemma \ref{lem3.4}  indicate that the random gradient models generated by Algorithm \ref{alg1} are probabilistically first-order accurate
if the smoothing parameter $\gamma_k$ and sampling size $b$ are chosen appropriately.   		

	\section{ Worst-case complexity}\label{sec5}

In this section, we %{\color{blue} follow the proof framework of \cite{WCC} to}
derive the worst-case bound of the iterations number $K_{\epsilon}$, for which
$\|\nabla f(X_{K_{\epsilon}})\| \leq \epsilon$, with high probability.
%We are  inspired by \cite{WCC}.
We first prove that the $k$-th iteration is successful if $\theta_k$ is large.
We make the following assumption.

		\begin{assumption}\label{ass:1}
(i)	$r(x)$ is Lipschitz continuous, i.e., there exists $\kappa_{lr}>0$ such that
		\begin{equation}\label{lipr}
			\|r(x)-r(y)\| \leq \kappa_{lr} \|x-y\|,\quad \forall x,y \in \mathbb{R}^n.
		\end{equation}
 (ii) For 	every		realization  of Algorithm \ref{alg1},
 $J_{m_k}$ is bounded  for all $k$, i.e., there exists
		$\kappa_{bj}>0$ such
		that
		\begin{equation}\label{ASS3}
			\|J_{m_k}\| \leq \kappa_{bj},\quad \forall k.
		\end{equation}

	\end{assumption}

	\begin{lemma}\label{lem3}
	Suppose that Assumptions \ref{ass:2} and \ref{ass:1}
		 hold.
		 	%				for every		realization  of Algorithm
		 	%\ref{alg1}.
		 	% and consider the $k$-th  %iteration of that realization.
	For $\eta \in (0,1)$,	if  (\ref{2.2realiz}) holds
			    %event $A_k$ happens
			with
		$\xi_1 \leq  \frac{\eta (1-p_0)}{2+ \eta (1-p_0)}$, and if
			\begin{equation}\label{lem3theta}
				\theta_k \ge
			 \frac{\frac{2\xi_2 a_1}{1-\xi_1}+
					\kappa_{bj}^2+  2(\kappa_{lr}^2+\kappa_{lj}\|r_0\|)  }{(1-\eta)(1-p_0) \|\nabla
					f_{m_k}(x_k)\|} ,
			\end{equation}
			then
			$\rho_k \ge p_0.$
	\end{lemma}

		\begin{proof}
		
		By Assumptions  \ref{ass:2} and \ref{ass:1}, for $t\in(0,1)$,
			\begin{align}\label{two}
				& \|J(x_k+ t d_k)^T r(x_k+td_k)-J(x_k)^T r(x_k) \|
				\nonumber\\
				&\leq  \|J(x_k+ t d_k)^T r(x_k+td_k) - J(x_k+ t d_k)^T
				r(x_k)  \|\nonumber\\
&\quad +  \|J(x_k+ t d_k)^T r(x_k)- J(x_k)^T r(x_k) \|
				\nonumber\\
				& \leq  \|J(x_k+ t d_k) \|\, \|r(x_k+td_k)-r(x_k) \| +  \|
				J(x_k+ t d_k)-J(x_k)  \|\, \|  r(x_k)  \| \nonumber\\
				& \leq \left(\kappa_{lr}^2+\kappa_{lj} \|r_0 \| \right) \|d_k \|.
			\end{align}
Hence,
				\begin{align*}
					 \left|\int_0^1 \left[ J(x_k+ t d_k)^T r(x_k+td_k)-J(x_k)^T
					r(x_k)					\right]^Td_k dt\right|
					  \leq  (\kappa_{lr}^2+\kappa_{lj}\|r_0\|) \|d_k\|^2.
				\end{align*}
						Thus, by Taylor's theorem,
		 \begin{align}\label{AP}
					&|Ared_k-Pred_k|=\left| \|r(x_k+d_k)\|^2-
					\|r(x_k)+ J_{m_k}
					d_k\|^2
					\right| \nonumber \\
					&= \left|    \|r(x_k)\|^2 + 2 \left(J_k^Tr(x_k) \right)^T d_k
					+ 2 \int_0^1\left[ J(x_k+ t d_k)^T
					r(x_k+td_k)-J_k^T r(x_k) \right]^Td_k dt \right.
					\nonumber\\
					& \left. \quad - \left( \|r(x_k)\|^2 +2
					(J_{m_k}^Tr(x_k))^Td_k +
					d_k^T
					J_{m_k}^T J_{m_k} d_k   \right) \right| \nonumber\\
				%	&=\left|2 (J_k^Tr(x_k))^Td_k + 2 \int_0^1 \left[ J(x_k+ t
%					d_k)^T
%					r(x_k+td_k)-J(x_k)^T r(x_k) \right]^Td_k dt  - 2
%					(\Tilde{J}_k^Tr(x_k))^Td_k - d_k^T \Tilde{J}_k^T \Tilde{J}_k
%					d_k
%					\right|
%					\nonumber\\
					&\leq  2 \left\| \left( (J_{m_k}- J_k)^Tr(x_k) \right)^T
					d_k
					\right\| +
					\left|
					d_k^T J_{m_k}^T J_{m_k} d_k  \right| \nonumber\\
& \quad+ 2\left|
					\int_0^1
					\left[ J(x_k+ t d_k)^T r(x_k+td_k)-J_k^T r(x_k)
					\right]^Td_k dt
					\right| \nonumber\\
					&\leq 2 \|\nabla f_{m_k}(x_k)- \nabla f(x_k)\|\,\|d_k\|+
					\|
					J_{m_k}^T
					J_{m_k} \|\,\|d_k\|^2 +
					2(\kappa_{lr}^2+\kappa_{lj}\|r_0\|) \|d_k\|^2 .
				\end{align}

			By (\ref{lem3theta}),
			\begin{align}\label{thetabig}
				\theta_k \ge
				\frac{\frac{2\xi_2 a_1}{1-\xi_1}+
					\kappa_{bj}^2+  2(\kappa_{lr}^2+\kappa_{lj}\|r_0\|)
					}{(1-\eta)(1-p_0) \|\nabla
					f_{m_k}(x_k)\|}  \ge \frac{\kappa_{bj}^2}{\|\nabla	
					f_{m_k}(x_k)\|}.
			\end{align}
This, together  with the definition of $d_k$  and
			(\ref{ASS3}), yields
			\begin{align}\label{3.16b}
				\|d_k\|\leq \frac{\|J_{m_k}^T
					r(x_k)\|}{\theta_k
				\|\nabla
				f_{m_k}(x_k)\|}\leq\frac{\|J_{m_k}^T
					r(x_k)\|}{\kappa_{bj}^2} \leq \frac{\|J_{m_k}^T
					r(x_k)\|}{\|J_{m_k}^T J_{m_k}\|}.
			\end{align}
Note that  by Powell's	result given  in \cite{POWELL}, we have
% and  the definition of $d_k$
			\begin{align}\label{pow}
				Pred_k \ge \|J_{m_k}^T r(x_k)\| \min \left\{ \|d_k\|,
				\frac{\|J_{m_k}^T r(x_k)\|}{\|J_{m_k}^T J_{m_k}\|}
				\right\}, %= \|J_{m_k}^T r(x_k)\|\|d_k\|,
				\quad
				\forall k.
			\end{align}
Hence,
			\begin{equation}\label{3.13}
				Pred_k \ge \|J_{m_k}^T r(x_k)\|\,\|d_k\|=\|\nabla
				f_{m_k}(x_k)\|\,\|d_k\|.
			\end{equation}
			Combining (\ref{2.2realiz}), (\ref{AP}) and \eqref{3.13}, we obtain
			\begin{align}\label{rhop0}
			 &	|\rho_k-1|= \frac{|Ared_k-Pred_k|}{Pred_k}\nonumber\\
			%	& \leq \frac{ 2 \|\nabla \tilde{f}(x_k)- \nabla f(x_k)\|
%				}{\|\nabla \tilde{f}(x_k)\|}
%				+\frac{\| \Tilde{J}_k^T \Tilde{J}_k
%					\|\,\|d_k\|}{\|\nabla \tilde{f}(x_k)\|} + \frac{
%						2(\kappa_{lr}^2+\kappa_{lj}\|r_0\|)\|d_k\|}{\|\nabla \tilde{f}(x_k)\|}
%				\nonumber\\
				&\leq \frac{1}{\|\nabla f_{m_k}(x_k)\|} \left[ 2
				\|\nabla f_{m_k}(x_k)- \nabla f(x_k)\| + \left( \|
				J_{m_k}^T
				J_{m_k} \| +  2(\kappa_{lr}^2+\kappa_{lj}\|r_0\|)
				\right)
				\|d_k\| \right] \nonumber\\
				& \leq  \frac{2\xi_1 \|\nabla
					f(x_k)\| }{\|\nabla f_{m_k}(x_k)\|} + \frac{2\xi_2
					\|d_{k-1}\|}{\|\nabla f_{m_k}(x_k)\|} +
				\frac{\left(\|J_{m_k}^T J_{m_k}\| +
				2(\kappa_{lr}^2+\kappa_{lj}\|r_0\|)
					\right)\|d_k\|}{\|\nabla f_{m_k}(x_k)\|}.
			\end{align}
			By	(\ref{2.2realiz}),
			\begin{equation}\label{3.16}
				(1-\xi_1)\|\nabla f(x_k)\| \leq \|\nabla f_{m_k}(x_k)\|
				+\xi_2
				\|d_{k-1}\|.
			\end{equation}
		Since  $\xi_1 \leq  \frac{\eta (1-p_0)}{2+ \eta (1-p_0)}<1$, we have
			\begin{align}\label{ft}
				\frac{\|\nabla f(x_k)\|}{\|\nabla f_{m_k}(x_k)\|}& \leq
				\frac{1}{1-\xi_1}+
				\frac{\xi_2 \|d_{k-1}\|}{(1-\xi_1)\|\nabla f_{m_k}(x_k)\|},\\
\frac{2\xi_1 }{1-\xi_1} &\leq \eta (1-p_0).
			\end{align}
			By the definition of $d_k$, we have
\begin{align}\label{3.19}
\|d_k\| \leq
			\frac{\|J_{m_k}^T
				r(x_k)\|}{\theta_k \|J_{m_k}^Tr(x_k)\|}=\frac{1}{\theta_{k}}.
\end{align}
Moreover, by (\ref{2.7a}), it always holds
\begin{align} \label{3.17}
\frac{1}{\theta_{k-1}}\leq \frac{a_1}{\theta_k}, \quad \forall k.
\end{align}
It then follows from \eqref{lem3theta}, \eqref{rhop0}--\eqref{3.17} that
			\begin{align}\label{3.14}
				|\rho_k-1| &\leq  \frac{2\xi_1 }{1-\xi_1} + \frac{2
					\xi_1\xi_2	\|d_{k-1}\|
				}{(1-\xi_1)\|\nabla f_{m_k}(x_k)\|} + \frac{2\xi_2
					\|d_{k-1}\|}{\|\nabla f_{m_k}(x_k)\|} \nonumber\\
&\quad+
				\frac{\left(\|J_{m_k}^T J_{m_k}\| +
				2(\kappa_{lr}^2+\kappa_{lj}\|r_0\|)
					\right)\|d_k\|}{\|\nabla f_{m_k}(x_k)\|}  \nonumber\\
				& = \frac{2\xi_1 }{1-\xi_1} + \frac{2\xi_2
					\|d_{k-1}\|}{(1-\xi_1)\|\nabla f_{m_k}(x_k)\|} +
				\frac{\left(\|J_{m_k}^T J_{m_k} \| +
				2(\kappa_{lr}^2+\kappa_{lj}\|r_0\|)
					\right)\|d_k\|}{\|\nabla f_{m_k}(x_k)\|}  \nonumber\\
				&\leq \frac{2\xi_1 }{1-\xi_1} + \frac{1}{\theta_{k} \|\nabla
					f_{m_k}(x_k)\| }\left(\frac{2\xi_2 a_1}{1-\xi_1}+
				\kappa_{bj}^2+  2(\kappa_{lr}^2+\kappa_{lj}\|r_0\|)  \right)\nonumber\\
 &\leq \eta (1-p_0) +(1-\eta)(1-p_0) \nonumber\\
 &=1-p_0.
			\end{align}
	This implies that $\rho_k \ge p_0$.			
		\end{proof}

		Define the set
\begin{align}\label{H1}
\mathcal{H}_1=\left\lbrace k\in \N: \rho_k \ge p_0,  \|\nabla
		f_{m_k}(x_k)\| \ge  \frac{p_1}{\theta_{k}}  \right\rbrace.
\end{align}
Denote by  $x^*$ the minimizer of \eqref{prob}.

   \begin{lemma}\label{lem5.2}
   	Suppose that Assumptions  \ref{ass:2} and \ref{ass:1} hold. For any realization
   	of Algorithm \ref{alg1},
   	\begin{align}\label{lem5.2eq}
   		\sum_{k=0}^{\infty} \frac{1}{\theta_k^2}\leq
   		\frac{a_2^{-2}}{1-a_1^{-2}} \left( 	\frac{p_0}{2}  \min\left\lbrace
   		\frac{p_1^2}{\kappa_{bj}^2},p_1\right\rbrace
   		\right)^{-1} \left( \|r_0\|^2 -\|r(x^*)\|^2\right):=\beta.
   	\end{align}
   \end{lemma}

\begin{proof}

By the definition of $d_k$ and Assumption \ref{ass:1},
	\begin{equation*}%\label{4.25}
	\|d_k\|\geq
	\frac{\|\nabla f_{m_k}(x_k)\|}{\kappa_{bj}^2+ \theta_k\|\nabla
	f_{m_k}(x_k)\|},
\end{equation*}
hence
\begin{align}\label{4.25}
\|d_k\|\geq
\begin{cases}
\dfrac{\|\nabla f_{m_k}(x_k)\|}{2\kappa_{bj}^2},	& \text{ if } \kappa_{bj}^2
\ge \theta_k\|\nabla f_{m_k}(x_k)\|, \\
\dfrac{1}{2\theta_k},	& \text{
otherwise. }
\end{cases}
\end{align}
It follows from
(\ref{ratio}), (\ref{pow}), (\ref{ASS3}) and (\ref{4.25})
that, for any
$k\in
\mathcal{H}_1$,
	\begin{align}\label{ut1}
&		\|r_k\|^2-\|r(x_k+d_k)\|^2 \nonumber\\
&\ge p_0 (\|r_k\|^2-\|r_k+J_{m_k}d_k\|^2) \nonumber\\
		& \ge p_0 \|\nabla f_{m_k}(x_k)\| \min \left\{
		\|d_k\|,
		\frac{\|\nabla f_{m_k}(x_k)\|}{\|J_{m_k}^T
		J_{m_k}\|}
		\right\} \nonumber\\
		&\ge p_0 \|\nabla f_{m_k}(x_k)\| \min \left\lbrace
		\min \left\{
		\frac{\|\nabla f_{m_k}(x_k)\|}{2\kappa_{bj}^2},
		\frac{1}{2\theta_k}
		\right\},
		\frac{\|\nabla f_{m_k}(x_k)\|}{\|J_{m_k}^T
			J_{m_k}\|}
		 \right\rbrace \nonumber\\
		 &\ge  p_0 \frac{p_1}{\theta_{k}}
		 \min \left\lbrace
		 \min \left\{
		 \frac{p_1}{2\kappa_{bj}^2\theta_{k}},
		 \frac{1}{2\theta_k}
		 \right\},
		 \frac{p_1}{\kappa_{bj}^2\theta_{k}}
		 \right\rbrace \nonumber\\
		& = p_0 \min\left\lbrace
		\frac{p_1^2}{\kappa_{bj}^2},p_1\right\rbrace \frac{1}{2 \theta_k^2}.
	\end{align}

	Summing over iteration indices no larger than $k$ in $\mathcal{H}_1$,
	we obtain
%	{\color{red}
		\begin{align*}
	\frac{p_0}{2}  \min\left\lbrace
\frac{p_1^2}{\kappa_{bj}^2},p_1\right\rbrace \sum_{
	j\in\mathcal{H}_1,	j\leq k}  \frac{1}{\theta_j^2}
&	\leq
\sum_{
	j\in\mathcal{H}_1,	j\leq k} 	(\|r_j\|^2-\|r_{j+1}\|^2)\nonumber\\
&\leq \sum_{j\leq k}(\|r_j\|^2-\|r_{j+1}\|^2) \\
&=\|r_0\|^2-\|r_{k+1}\|^2 \\
&\leq \|r_0\|^2 -\|r(x^*)\|^2,	
\end{align*}
%	\begin{align*}
%		\frac{p_0}{2}  \min\left\lbrace
%		\frac{p_1^2}{\kappa_{bj}^2},p_1\right\rbrace \sum_{
%			k\in\mathcal{H}_1}  \frac{1}{\theta_k^2}
%	&	\leq
%		 \sum_{
%			k\in\mathcal{H}_1} 	(\|r_k\|^2-\|r_{k+1}\|^2)\nonumber\\
%	&\leq \sum_{k\in\mathcal{H}_1}(\|r_k\|^2-\|r_{k+1}\|^2) \\
%	&=\|r_0\|^2-\|r_{k+1}\|^2 \\
%&\leq \|r_0\|^2 -\|r(x^*)\|^2,
%	\end{align*}
%}
	thus
	\begin{equation}\label{5.4}
		\sum_{k\in\mathcal{H}_1 } \frac{1}{\theta_k^2} \leq \left( 	
		\frac{p_0}{2}  \min\left\lbrace
		\frac{p_1^2}{\kappa_{bj}^2},p_1\right\rbrace \right)^{-1} \left(
		\|r_0\|^2
		-\|r(x^*)\|^2\right) .
	\end{equation}
	
If $\mathcal{H}_1$ is infinite, denote $\mathcal{H}_1=\{k_1,k_2,\ldots\}$.
For auxiliary reasons, set $k_0=-1$ and 	$\frac{1}{\theta_{-1}}=0$.
	By (\ref{2.7a}),
	\begin{align*}
		\theta_k \geq a_1^{k-k_i-1}a_2\theta_{k_i}, \quad k=k_i+1,\ldots,k_{i+1},
	\end{align*}
	which yields
	\begin{equation*}
		%\sum_{i=0}^{\infty} \sum_{k=k_i+1}^{k_{i+1}} \|d_k\|^2 \leq
		\sum_{k=k_i+1}^{k_{i+1}} \frac{1}{\theta_k^2} \leq
		\frac{a_2^{-2}}{1-a_1^{-2}} \frac{1}{\theta_{k_i}^2}.
	\end{equation*}
	So by \eqref{5.4},
	\begin{align*}\sum_{k=0}^{\infty}\frac{1}{\theta_k^2} &=
	\sum_{i=0}^{\infty} \sum_{k=k_i+1}^{k_{i+1}} \frac{1}{\theta_k^2}
	\nonumber\\
 &\leq
		\frac{a_2^{-2}}{1-a_1^{-2}} \sum_{i=0}^{\infty}
		\frac{1}{\theta_{k_i}^2}= \frac{a_2^{-2}}{1-a_1^{-2}} \sum_{k\in
		\mathcal{H}_1} \frac{1}{\theta_k^2} \\
		&\leq \frac{a_2^{-2}}{1-a_1^{-2}} \left( 	\frac{p_0}{2}
		\min\left\lbrace
		\frac{p_1^2}{\kappa_{bj}^2},p_1\right\rbrace \right)^{-1} \left(
		\|r_0\|^2
		-\|r(x^*)\|^2\right).
	\end{align*}
	If  $\mathcal{H}_1$ is finite, the 	same bound could be derived by the similar analysis.
	\end{proof}

Based on Lemma \ref{lem3} and Lemma \ref{lem5.2}, we have the following result.

\begin{lemma}\label{lem5.1}
	Suppose that Assumptions  \ref{ass:2} and \ref{ass:1} hold.	For  $\eta \in (0,1)$,
	if  (\ref{2.2realiz}) holds		 with
	$\xi_1 \leq  \frac{\eta (1-p_0)}{2+ \eta (1-p_0)}$,  and if
 	\begin{equation}\label{lem5.1theta}
		\frac{1}{	\theta_k} \leq  \kappa_{\theta}\|\nabla f(x_k)\|,
	\end{equation}
	where $\kappa_{\theta}=\frac{1-\xi_1}{\xi_2a_1+ c_{\theta}}$ with
		$c_{\theta}=\max \left\lbrace  \frac{\frac{2\xi_2 a_1}{1-\xi_1}+
		\kappa_{bj}^2+  2(\kappa_{lr}^2+\kappa_{lj}\|r_0\|)
		}{(1-\eta)(1-p_0)},   p_2 \right\rbrace$, 	
	then
	$\rho_k \ge p_0$ and $\theta_k$ is   decreased.   %not
	%increased.
\end{lemma}

\begin{proof}
	Since $\|d_k\| \leq \frac{1}{\theta_k} $ and $ \theta_k \leq a_1\theta_{k-1}$,  by (\ref{2.2realiz}),
	\begin{align*}
		\| \nabla f_{m_k} (x_k) -\nabla f(x_k)\| \leq \xi_1 \|\nabla f(x_k)\|+
		\frac{\xi_2 a_1}{\theta_k},
	\end{align*}
	which implies that
	\begin{align}\label{bu1}
		-  \frac{	\xi_2 a_1}{\theta_k} \leq \xi_1 \|\nabla f(x_k)\|-\|
		\nabla f_{m_k} (x_k) -\nabla f(x_k)\| .
	\end{align}	
By (\ref{lem5.1theta}),
	\begin{align}\label{5.7}
		  \frac{\xi_2a_1+ c_{\theta}}{	\theta_k} \leq  (1-\xi_1)\|\nabla
		f(x_k)\|.
	\end{align}
Combining \eqref{bu1} and \eqref{5.7}, we have
	\begin{align*}
		 \frac{c_{\theta}}{\theta_k} \leq  (1-\xi_1)\|\nabla f(x_k)\|+ \xi_1
		\|\nabla f(x_k)\|-\| \nabla f_{m_k} (x_k) -\nabla f(x_k)\|\leq\|\nabla
		f_{m_k} (x_k)\|.
	\end{align*}
	Therefore
	\begin{align}\label{bu2}
		\theta_k \ge \frac{c_{\theta}}{\|\nabla f_{m_k} (x_k)\|}
	  \ge
		\frac{p_2}{\|\nabla f_{m_k} (x_k)\|}  .
	\end{align}
		By Lemma \ref{lem3} and \eqref{2.7a}, $\rho_k\geq p_0$ and $\theta_k$
		is   decreased.   %not increased.
\end{proof}

Define the set
\begin{align}
\mathcal{H}_2=\left\lbrace k\in \N: \rho_k \ge p_0,  \|\nabla
f_{m_k}(x_k)\|  \ge  \frac{p_2}{\theta_{k}}  \right\rbrace.
\end{align}
	Define
\begin{align}
	y_k=\begin{cases}
	1,	& \text{if}\ k\in\mathcal{H}_2,\\
	0,	& \text{otherwise},
	\end{cases}
\end{align}
and
\begin{align}\label{zl}
	z_k=\begin{cases}
		1,	& \text{if  (\ref{2.2realiz}) holds		 with
			$\xi_1 \leq  \frac{\eta (1-p_0)}{2+ \eta (1-p_0)}$},   \\
		0,	& \text{otherwise}.
	\end{cases}
\end{align}
Denote $\hat k$ as the smallest number in $[0,\ldots,k]$
such that
 $$\|\nabla f(x_{\hat{k}})\|=\min_{l\in[0,\ldots,k]}\|\nabla f(x_l)\|.
 $$

\begin{lemma}\label{lem5.3}
	Suppose that Assumptions  \ref{ass:2} and \ref{ass:1} hold. For any realization
	of Algorithm \ref{alg1} and $k$,
	\begin{align}\label{lem5.3eq}
		\sum_{l=0}^{k-1}   z_l <
		\frac{\beta}{(\min\{a_2\theta_0^{-1},\kappa_{\theta}\|\nabla f
		(x_{\hat{k}})\|\})^2} + \frac{\ln(a_1)}{\ln(a_1/a_2)} k.
	\end{align}
%	where $c=\frac{\ln(a_1)}{\ln(a_1/a_2)}$.

\end{lemma}			
	\begin{proof}
	
	For each $l\in[0,\cdots,k-1]$, define
	\begin{align}\label{5-3-1}
		v_l=\begin{cases}
			1,	& \text{if  } \frac{1}{\theta_l}< \min\{a_2 \theta_0^{-1},\kappa_{\theta}\|\nabla f (x_{\hat{k}})\|\},  \\
			0,	& \text{otherwise}.
		\end{cases}
	\end{align}	
	When $v_l=1$, if $z_l=1$,  then by Lemma \ref{lem5.1}, we have $y_l=1$, hence
	\begin{equation}\label{a1}
		z_l \leq (1-v_l)+v_l y_l.
	\end{equation}
Note that \eqref{a1} also holds when $v_l=0$ because $z_l \leq 1$ for all $l$.
So \eqref{a1} holds true for all $l$.

    By	 (\ref{5-3-1}),
    \begin{equation*}
    	1-v_l \leq \frac{\theta_l^{-2}}{(\min\{a_2
    	\theta_0^{-1},\kappa_{\theta}\|\nabla f (x_{\hat{k}})\|\})^2}.
    \end{equation*}
It then follows from Lemma \ref{lem5.2} that
	\begin{align}\label{a2}
		\sum_{l=0}^{k-1} (1-v_l) \leq
		\frac{\beta}{(\min\{a_2\theta_0^{-1},\kappa_{\theta}\|\nabla f
		(x_{\hat{k}})\|\})^2}.
	\end{align}
	
		If $v_l y_l=1$, then $y_l=1$ and $v_l=1$. These imply that  $\frac{1}{\theta_l}< a_2\theta_0^{-1}
		\leq a_2\theta_{\min}^{-1} $, hence
		$\theta_{l+1}=\max\{a_2\theta_l, \theta_{\min}\}=a_2 \theta_l$.
Note that $\theta_{l+1} \leq  a_1 \theta_{l}$ holds for all $l$. Hence,
	\begin{align}\label{5-3-2}
		\theta_k \leq \theta_0 \prod_{l=0}^{k-1} a_1^{1-v_ly_l}  a_2^{v_ly_l}.
	\end{align}

	Denote $\bar{k}=\max\{l:v_l=1, l\in[0,\ldots,k-1] \}$.
		Since $v_{\bar{k}}=1$, by (\ref{5-3-1}), $\frac{1}{\theta_{\bar{k}}}<
a_2\theta_0^{-1} $, then
$\theta_{\bar{k}+1} > \theta_0$
because $\theta_{l+1} \ge a_2 \theta_l$ for all $l$.
	Therefore, by (\ref{5-3-2}),
		\begin{align*}
 	\theta_0  <	\theta_{\bar{k}+1}  \le
\theta_0 \prod_{l=0}^{\bar{k}}a_1^{1-v_ly_l}a_2^{v_ly_l},
	\end{align*}
	which leads to
	\begin{align*}
 	0 <   \sum_{l=0}^{\bar{k}} (\ln(a_1)(1-v_ly_l)
+\ln(a_2)v_ly_l)=(\bar{k}+1)\ln(a_1)+\sum_{l=0}^{\bar{k}}v_ly_l
\ln(a_2/a_1).
	\end{align*}
	Then
	\begin{equation}\label{c2}
 	\sum_{l=0}^{\bar{k}}v_ly_l <
		\frac{\ln(a_1)}{\ln(a_1/a_2)}(\bar{k}+1)
		\leq\frac{\ln(a_1)}{\ln(a_1/a_2)}k.
	\end{equation}
	Therefore
	\begin{equation}\label{a3}
	 \sum_{l=0}^{k-1} v_ly_l=\sum_{l=0}^{\bar{k}}   v_ly_l <
	 \frac{\ln(a_1)}{\ln(a_1/a_2)} k.
	\end{equation}

Combining (\ref{a1}), (\ref{a2}) and (\ref{a3}), we obtain (\ref{lem5.3eq}).
		\end{proof}

	\begin{theorem}\label{the5.5}
		
	Suppose that Assumptions  \ref{ass:2} and \ref{ass:1} hold.  For $0< \epsilon
	\leq  \frac{a_2}{\kappa_{\theta} \theta_0}$ and $ \eta \in (0,1)$,
%		if 	$\nabla f_{M_k}(X_k)$  generated by 		Algorithm \ref{alg1}
%are 			
%	$\alpha$-probabilistically $(\xi_1,\xi_2)$-first-order accurate
%	with  $\alpha> \frac{\ln(a_1)}{\ln(a_1/a_2)}$ and
%	$\xi_1 \leq  \frac{\eta (1-p_0)}{2+ \eta (1-p_0)}$,  and if
take
$\gamma_k=\|d_{k-1}\|$, $b\ge n\Big/\left(
\frac{\eta^2(1-p_0)^2}{2 \left( 2+\eta
	(1-p_0)\right) ^2} +1\right)$.
If
	\begin{equation}\label{the5-1}
%		k \ge \frac{2\beta}{ \left(
%				\alpha-\frac{\ln(a_1)}{\ln(a_1/a_2)}\right)\kappa_{\theta}^2
%				\epsilon^2},
	k \ge \frac{2\beta}{ \left(
	\frac{1}{2}-\frac{\ln(a_1)}{\ln(a_1/a_2)}\right)\kappa_{\theta}^2
	\epsilon^2},
	\end{equation}
	then
	\begin{align}\label{the5eq}
%		\P \left(  \|\nabla f(X_{\hat{k}})\| \leq
%		\left(\frac{2\beta}{ \left(
%				\alpha-\frac{\ln(a_1)}{\ln(a_1/a_2)}\right)\kappa_{\theta}^2
%				k}\right)^{1/2}\right)
%		\ge  1- \exp\left( -\frac{(\alpha-\frac{\ln(a_1)}{\ln(a_1/a_2)})^2}{8
%		\alpha}k \right).
	\P \left(  \|\nabla f(X_{\hat{k}})\| \leq
\left(\frac{2\beta}{ \left(
	\frac{1}{2}-\frac{\ln(a_1)}{\ln(a_1/a_2)}\right)\kappa_{\theta}^2
	k}\right)^{1/2}\right)
\ge  1- \exp\left( -\frac{(\frac{1}{2}-\frac{\ln(a_1)}{\ln(a_1/a_2)})^2}{4}k
\right).
	\end{align}

	\end{theorem}

\begin{proof}
	By remark \ref{remark}, $\{\nabla
	f_{M_k}(X_k)\}$ generated 		 by 	 		 Algorithm \ref{alg1}
	are	
	$\frac{1}{2}$-probabilistically
	$(\xi_1,\xi_2)$-first-order 	accurate with
	$\xi_1 \leq  \frac{\eta (1-p_0)}{2+ \eta (1-p_0)}$.
	Since $a_1>1>a_2>0$ and $a_1a_2<1$, it can be proved that
	$\frac{\ln(a_1)}{\ln(a_1/a_2)}<
	\frac{1}{2}$.

Let $\alpha=\P (Z_l=1)$, then $\mathbb{E}[\sum_{l=0}^{k-1}Z_l ]=\alpha k$.
%From above we know that $\alpha=\frac{1}{2}$.
 The multiplicative Chernoff bound shows that the lower tail of $\sum_{l=0}^{k-1}Z_l$ satisfies
	 \begin{align}\label{5.17}
	 	\P\left( \sum_{l=0}^{k-1}Z_l \leq (1-\sigma)\alpha k \right) \leq
	 	\exp\left(- \frac{\alpha \sigma^2}{2}k\right), \quad \forall
	 	1>\sigma>0.
	 \end{align}	
For $\nu \in (0,\alpha)$, define
	\begin{align}
			\pi_k(\nu) := \P\left( \sum_{l=0}^{k-1}Z_l \leq \nu k\right).
\end{align}
Setting $\sigma= \frac{\alpha- \nu}{\alpha}$ in \eqref{5.17}, we obtain
	\begin{align}\label{5-4-1}
			\pi_k(\nu)   \leq \exp\left(- \frac{(\alpha-\nu)^2}{2\alpha}k \right).
		\end{align}	
		
By Lemma \ref{lem5.3}, for any given $\epsilon$ satisfying
	\begin{equation}\label{m1}
			0< \epsilon \leq  \frac{a_2}{\kappa_{\theta} \theta_0},
	\end{equation}
    we have
	\begin{align}\label{m2}
		\P \left(\|\nabla f(X_{\hat{k}})\|  \ge \epsilon \right)
	%	&  \leq \P \left(   \sum_{l=0}^{k-1} Z_l < \frac{\beta}{
%			\kappa_{\theta}^2 \epsilon^2}+ \frac{\ln(a_1)}{\ln(a_1/a_2)} k
%		\right)    \nonumber\\
		&\leq \P \left(   \sum_{l=0}^{k-1} Z_l \leq \frac{\beta}{
		\kappa_{\theta}^2 \epsilon^2}+ \frac{\ln(a_1)}{\ln(a_1/a_2)} k
		\right).
	\end{align}
 Since $\alpha=\frac{1}{2}> \frac{\ln(a_1)}{\ln(a_1/a_2)}$,
for any
	\begin{equation*}%\label{k1}
		k \ge \frac{2\beta}{ \left(
			\frac{1}{2}-\frac{\ln(a_1)}{\ln(a_1/a_2)}\right)\kappa_{\theta}^2
			\epsilon^2}
		= \frac{2\beta}{\left( \alpha-
			\frac{\ln(a_1)}{\ln(a_1/a_2)}\right) \kappa_{\theta}^2 \epsilon^2},
	\end{equation*}
 by \eqref{m1},
\begin{align}\label{5.23}
0<\frac{\beta}{ \kappa_{\theta}^2
		\epsilon^2 k}+\frac{\ln(a_1)}{\ln(a_1/a_2)} \leq
		\frac{\alpha+\frac{\ln(a_1)}{\ln(a_1/a_2)}}{2}<\alpha.
\end{align}
Note that $\pi_k(\cdot )$ is a monotone increasing function.
		By \eqref{5-4-1}, (\ref{m2})--\eqref{5.23},
		\begin{align}\label{key1}
\P\left( \|\nabla f(X_{\hat{k}})\| \leq \epsilon  \right)
&\ge
		1-\P\left( \sum_{l=0}^{k-1} Z_l \leq \frac{\beta}{ \kappa_{\theta}^2
		\epsilon^2}+\frac{\ln(a_1)}{\ln(a_1/a_2)} k   \right)  \nonumber\\
		& = 1- \pi_k\left(\frac{\beta}{ \kappa_{\theta}^2
		\epsilon^2 k}+\frac{\ln(a_1)}{\ln(a_1/a_2)}  \right)\nonumber\\
&		\ge 1-\pi_k\left(
		\frac{\alpha+\frac{\ln(a_1)}{\ln(a_1/a_2)}}{2}
		\right)
		 \nonumber\\
		& \ge  1- \exp\left(
		-\frac{\left(\alpha-\frac{\ln(a_1)}{\ln(a_1/a_2)}\right)^2}{8 \alpha}k
		\right),\nonumber\\
&=	1- \exp\left(
		-\frac{(\frac{1}{2}-\frac{\ln(a_1)}{\ln(a_1/a_2)})^2}{4}k
		\right).	
	\end{align}
which gives \eqref{the5eq} when the equality holds in \eqref{the5-1}.
	\end{proof}

%		Let
%		\begin{equation*}
%			\epsilon =\frac{(2\beta)^{1/3}}{\kappa_{\theta} \left(
%				\alpha-\frac{\ln(a_1)}{\ln(a_1/a_2)}\right) ^{1/3}}k^{-1/3}.
%		\end{equation*}

Based on the above, 	we can give a worst-case bound of $O(\epsilon^{-2})$
for
	$K_{\epsilon}$ with high probability. Notice that  $K_{\epsilon}$ is a
	random variable and $\P\left( K_{\epsilon} \leq k\right) =\P\left(
	\|\nabla f(X_{\hat{k}})\|  \leq \epsilon \right)  $.	
Applying  (\ref{key1}), we have the following result.
	
	\begin{theorem}\label{last}
 	Suppose that Assumptions  \ref{ass:2} and \ref{ass:1} hold.
 For $0< \epsilon \leq  \frac{a_2}{\kappa_{\theta} \theta_0}$ and $ \eta \in
 (0,1)$,
 take
 $\gamma_k=\|d_{k-1}\|$ and $b\ge n\Big/\left(
 \frac{\eta^2(1-p_0)^2}{2 \left( 2+\eta
 	(1-p_0)\right) ^2} +1\right)$, then
		 \begin{align}%\label{the5.6}
		 	\P \left( K_{\epsilon} \leq \left \lceil   \frac{2\beta}{\left(
		 		\frac{1}{2}-\frac{\ln(a_1)}{\ln(a_1/a_2)}\right)
		 		\kappa_{\theta}^2
		 		\epsilon^2}  \right \rceil  \right)
\ge 1-\exp \left(-
		 		\frac{\left(
		 	\frac{1}{2}-\frac{\ln(a_1)}{\ln(a_1/a_2)}\right) \beta}{2
		 		\kappa_{\theta}^2 \epsilon^2}\right) .
		 \end{align}
	\end{theorem}

	\section{ Numerical experiments}\label{sec6}

%	We know from Section \ref{sec5} that our algorithm has lower complexity
%	when taking  $b=n$ and $a_1 \ge (1+\sqrt{5})/2$, so in the following
%	experiments, we take $b=n$ and
%	$a_1=4$.
%	We test  Algorithm \ref{alg1} (DFLM-OSS) and compare it with the a
%	derivative-free LM method
%	that uses coordinate-wise  finite difference to approximate Jacobian
%	(DFLM-FD).
	We test  Algorithm \ref{alg1} on some nonlinear least squares problems.
 The experiments are implemented on a laptop with an Intel(R) Core(TM) i7-10710U CPU (1.61GHz) and 16.0GB of RAM, using Matlab R2020b.

	We set $p_0=0.001, p_1=0.25, p_2=0.75, a_1=4, a_2=0.25,
		\theta_0=10^{-8}, \theta_{\min}=10^{-8}, \epsilon_0=10^{-4}$, $b=n$ for all tests.
		%and {\color{brown} 	$\gamma_k$ as in theorems for all tests. }
	%	$\gamma_k=\|d_{k-1}\| $.
 The algorithm is terminated when the norm of $J_{m_k}^Tr(x_k)$ is
  less
 	than $\epsilon_0$   or the iteration number exceeds
		$1000(n+1)$.
		The results are listed in the tables.
``Niter" and ``NF" represent the numbers of iterations and function
calculations, respectively.
%; ``time" represents the running time.
%and “n.s.x∗?” gives a Y(yes) if the method converges to
%the same solution as the corresponding nonsingular problem, an N(no) otherwise.
%If the algorithm failed to find the solution in $1000(n + 1)$ iterations, we
%denoted it by
%the sign ``–”, and
If the iterations had underflows or overflows, we denoted it by OF.
%		``f\_end" denotes		the	$f$ 	values generated by algorithms,
%``Error\_x" is the norms of
%		the	difference	between the solutions generated by algorithms and the
%		true		solutions.
DFLM-FD represents Algorithm \ref{alg1} with  $J_{m_k}$  being computed by
		\begin{equation}\label{FD}
			[J_{m_k}]_j=\frac{r(x_k+\gamma_k e_j)-r(x_k)}{\gamma_k},\quad
			j=1,\ldots,n,
		\end{equation}
		where $[J_{m_k}]_j$ is the $j$-th column of  $ J_{m_k}$.
%		 (cf. \cite{DFOLM}).
%		We also compare DFLM-OSS with a LM method based on 	probabilistic
%		Jacobian models called LM-PJM\cite{zhao}, whose approximated  Jacobian
%		is built by linear interpolation.
DFLM-OSS represents Algorithm \ref{alg1} with  $J_{m_k}$  being computed by (\ref{JA}) and (\ref{nabr}).
In DFLM-OSSv1, we generate $b$ independent random
		vectors $w_1,\ldots,w_b$ with identical distribution $\mathcal{N}(0,I)$ at each iteration,
then compute the QR decomposition of the matrix  $[w_1,\ldots,w_b]$ to obtain $u_1,\ldots,u_b$.
%In DFLM-OSSv2,
%		represents Algorithm \ref{alg1} with $\{u_j\}$ given by
%		QR decomposition of
%		$b/2$ random
%		coordinate directions and $b/2$    independent random
%				vectors with identical distribution	$\mathcal{N}(0,I)$.
%		 random orthogonal  directions
%		obtained by using  Gram-Schmidt orthogonalization
%		that are orthogonal to the selected coordinate direction  and
%orthogonal to each other.
In DFLM-OSSv2, we first generate 10 orthogonal direction sets,
then choose one of them randomly as directions at each iteration.

%	{\color{brown}
		For each problem, DFLM-OSSv1
		and DFLM-OSSv2 run 60 times and give the average iteration number and
		function
		calculations.
%}

%%%%%%%%%%%%%%%%%%%%%%%%%%%%%%%%%%%%%%%%%%%%%%%%%%%%%%%%%%%%%%%		

	\begin{example}[\cite{Banderia2014}]\label{exam4.2}
Consider  \eqref{prob} with
	$$
	\begin{matrix}r_1\left(x\right)=100\left(x_1-x_2^2\right)^2
		+\left(1-x_2\right)^2,\\r_2\left(x\right)=100\left(x_2-x_3^2\right)^2+
		\left(1-x_3\right)^2,\\r_3\left(x\right)=100\left(x_3-x_1^2\right)^2+
		\left(1-x_1\right)^2.\\\end{matrix}
	$$
	The optimal solution of this problem is
	$\left(x_1^\ast,x_2^\ast,x_3^\ast\right)=\left(1,1,1\right)^T$.

The initial
point
is chosen as $x_0=10v$, where $v$ is generated by the multivariate
Gaussian distribution $\mathcal{N}(\mathbf{0},\mathbf{I})$.
	 We  apply DFLM-FD and DFLM-OSS. The results are given as follows.
	
	\begin{table}[H]
	\caption{} \label{table1}	
	\centering
		\begin{tabular}{lllll}
	\hline
	\multirow{2}{*}{$n$} & DFLM-FD        & DFLM-OSSv1      & DFLM-OSSv2      \\
	& Niter/NF   & Niter/NF  & Niter/NF    \\
	\hline
3                  & 41/204 & 35/163  & 29/139 \\ \hline
\end{tabular}
\end{table}

\end{example}

		\begin{example}[\cite{Powell2003}]\label{exam4.4}
			Consider   \eqref{prob} with
			\begin{align*}
				r_i\left(x\right)&=100\left(\left(x_i^2+x_n^2\right)^2		
-4x_i+3\right),\quad i=1,\ldots,n-1,\\r_n\left(x\right)&=100x_n^4.
\end{align*}
			The optimal solution is $x^\ast=\left(1,\ldots,1,0\right)^T$.

The initial
			point 			is chosen as $x_0=10v$, where $v$ is generated by the multivariate
			Gaussian distribution $\mathcal{N}(\mathbf{0},\mathbf{I})$. We
%			modify this test problem by (\ref{modi}) and (\ref{5.2})
%			and then
			 apply
			DFLM-FD and   DFLM-OSS.
The results are given
			in the following table.
	\begin{table}[H]
	\caption{} \label{table2}
		\centering
		\begin{tabular}{lllll}
	\hline
	\multirow{2}{*}{$n$} & DFLM-FD          & DFLM-OSSv1        &
	DFLM-OSSv2
	     \\
	& Niter/NF      & Niter/NF         &
	Niter/NF
	\\ \hline
30                 & 109/3431 & 83/2631  & 81/2559  \\
50                 & 148/7617  & 96/4974  & 97/4979  \\
\hline
\end{tabular}
\end{table}

		\end{example}

\begin{example}[\cite{more}] \label{286} Consider   \eqref{prob} with
	\begin{align*}
		r_i\left(x\right)=&10\left(x_i^2-x_{i+10}\right),\\
		r_{i+10}\left(x\right)=&x_i-1,  \quad	i=1,\ldots,10.
	\end{align*}
	%This is the 286-th text example in \cite{more}.
The optimal solution is
	$x^\ast=\left(1,\ldots,1\right)^T$.

We take	 	the initial 	point
 $x_0=10v$, where $v$ is generated by the multivariate
	Gaussian distribution $\mathcal{N}(\mathbf{0},\mathbf{I})$.
%	
%	initial point  $x_0=\left(-1.2,\ldots,-1.2,1,\ldots,1   \right)^T$ whose
%	half of
%	the elements are $-1.2$ and the remaining  are $1$.
	We  apply DFLM-FD and
	DFLM-OSS. The 	results are given in the 	following table.	
	\begin{table}[H]
	\caption{} \label{table3}
%		\centering
		\begin{tabular}{lllll}
	\hline
	\multirow{2}{*}{$n$} & DFLM-FD         & DFLM-OSSv1       &
	DFLM-OSSv2          \\
	& Niter/NF     & Niter/NF    & Niter/NF
	\\
	\hline
20                 & 115/2473 & 88/1900 & 140/3016  \\
\hline
\end{tabular}
\end{table}
%	It shows that in this test the two versions of DFLM-OSS have less NF and
%	number of iterations than DFLM-FD and DFLM-OSSv2 has a running time close
%	to that of DFLM-FD.

\end{example}

		\begin{example}(\cite{More81})
 	Consider   \eqref{prob} with
	\begin{align*}
	r_i\left(x\right)=&10^{-5/2}(x_i -1), \quad	i=1,\ldots,n.\\
	r_{n+1}=&\left( \sum_{i=1}^{n} x_i\right)-\frac{1}{4}     .
\end{align*}
The optimal value is $f^*=7.08765\times 10^{-5}$.

We take	$n=10$ and	 $(x_0)_i=i (i=1,\ldots,n).$
The initial points are taken as $x_0, 10x_0$ and $100x_0$.
We  apply DFLM-FD and
DFLM-OSS. The 	results are given in the 	following table.
	\begin{table}[H]
	\caption{} \label{table3}
	%		\centering
	\begin{tabular}{llllll}
		\hline
	\multirow{2}{*}{$n$}  & \multirow{2}{*}{$x_0$} &
	DFLM-FD            & DFLM-OSSv1                    &
	DFLM-OSSv2         \\                    &                     &
	Niter/NF        & Niter/NF         &
	Niter/NF
		\\
		\hline
		\multirow{3}{*}{10} & 1   & 18/215 & 18/208 &
		17/206 \\
		& 10  & 22/258 & 24/281 & 25/297 \\
		& 100 & 30/354 & 32/384 & 32/384 \\ \hline
	\end{tabular}
\end{table}

\end{example}
%%%%%%%%%%%%%%%%%%%%%%%%%%%%%%%%%%%%%%%%%%%%%%%%%%%%%%%%%%%%%%%%%%	
	
\begin{example}

	%		\subsection{Performance profile}
	Consider the nonlinear equations  created by modifying the standard
	nonsingular
	test functions 	given by  	 Mor\'{e}, Garbow and Hillstrom in
	\cite{More81}.	They have the  form 	
	\begin{equation}\label{modi}
		\hat{r}(x)=r(x)-J(x^*)A(A^TA)^{-1}A^T(x-x^*)
	\end{equation}
	as in \cite{ad11}, where $r(x)$ is the standard nonsingular test
	function,
	$x^*$ is its
	root,
	and $A\in R^{n\times k}$ has full column rank with $1\leq k\leq n$.
	Obviously, $\hat{r}(x^*)=0$ and
	$$
	\hat{J}(x^*)=J(x^*)(I-A(A^TA)^{-1}A^T)
	$$
	has rank $n-k$.  We take
	\begin{equation}\label{5.2}
		A\in R^{n\times 1},\qquad A^T=(1,1,\cdots, 1).
	\end{equation}

	The results are given in the following table.
	The first column is the problem number in ``Systems of Nonlinear
	Equations" (cf. \cite{More81}).
	The third column  indicates that the starting
	point is $x_0, 10x_0$, and $100x_0$, where $x_0$ is suggested in
	\cite{More81}.

	\begin{table}[H]
		\caption{} \label{table5}
		\begin{tabular}{llllll}
			\hline
			\multirow{2}{*}{Prob} & \multirow{2}{*}{$n$}  &
			\multirow{2}{*}{$x_0$} &
			DFLM-FD            & DFLM-OSSv1                    &
			DFLM-OSSv2         \\
			&                     &                     &
			Niter/NF        & Niter/NF         &
			Niter/NF              \\ \hline
			\multirow{3}{*}{1}                        &
			\multirow{3}{*}{2}                     & 1                   &
			18/71        & 14/53     & 14/52       \\
			&                                        & 10
			&
			86/320       & 24/93      & 35/136      \\
			&                                        & 100
			&
			142/559      & 46/176     & 55/211     \\ \hline
			\multirow{3}{*}{8}                        &
			\multirow{3}{*}{50}                    & 1                   &
			7/363        & 8/399     & 8/398       \\
			&                                        & 10
			&
			87/4523     & 87/4527    & 85/4409     \\
			&                                        & 100
			&
			201/10451    & OF               & OF                \\ \hline
			\multirow{3}{*}{9}                        &
			\multirow{3}{*}{50}                    & 1                   &
			2/103       & 2/103      & 2/103       \\
			&                                        & 10
			&
			3/155        & 4/210     & 4/211       \\
			&                                        & 100
			&
			10/519      & 10/531     & 10/515     \\ \hline
			\multirow{3}{*}{10}                       &
			\multirow{3}{*}{50}                    & 1                   &
			8/411        & 4/228     & 5/246       \\
			&                                        & 10
			&
			7/363      & 7/343     & 7/377      \\
			&                                        & 100
			&
			257/13358  & 11/559     & 25/1290    \\ \hline
			\multirow{3}{*}{11}                       &
			\multirow{3}{*}{50}                    & 1                   &
			12/623       & 35/1809    & 39/2001    \\
			&                                        & 10
			&
			60/3104     & 47/2448    & 52/2677     \\
			&                                        & 100
			&
			46/2373     & 53/2738    & 55/2874     \\ \hline
			\multirow{3}{*}{12}                       &
			\multirow{3}{*}{50}                    & 1                   &
			16/831      & 19/972     & 18/925     \\
			&                                        & 10
			&
			22/1143      & 27/1377   & 39/2005     \\
			&                                        & 100
			&
			3409/175561  & 958/49339 & 1312/67579  \\ \hline
			\multirow{3}{*}{13}                       &
			\multirow{3}{*}{50}                    & 1                   &
			13/669      & 11/564    & 11/558     \\
			&                                        & 10
			&
			1176/60615   & 192/9950  & 176/9127    \\
			&                                        & 100
			&
			192/9982    & 225/11692  & 237/12292  \\ \hline
			\multirow{3}{*}{14}                       &
			\multirow{3}{*}{50}                    & 1                   &
			25/1284      & 13/681    & 13/651     \\
			&                                        & 10
			&
			13/675       & 58/3026    & 69/3546     \\
			&                                        & 100
			&
			23/1195      & 108/5564   & 118/6109    \\ \hline
			\multirow{3}{*}{23}                       &
			\multirow{3}{*}{10}                    & 1                   &
			94/1124      & 86/1028    & 83/989     \\
			&                                        & 10
			&
			87/1038      & 96/1147    &100/1194     \\
			&                                        & 100
			&
			104/1241      & 104/1241   & 109/1305    \\ \hline
		\end{tabular}
	\end{table}
	
\end{example}	

%{\color{blue}
	
	In the following we present the numerical results  using the
	performance profile developed by Dolan, Moré  and
	Wild\cite{Dolan2002}.
		Denote by $\mathcal{S}$ the set of solvers and $\mathcal{P}$ the set of test problems, respectively.
	If a solver $s\in{\mathcal{S}}$  gives a point
	$x$ that satisfies
	\begin{align}\label{convtest}
		|f(x)-f_p^*|\leq
		\tau,
	\end{align}
	where $\tau\in(0,1)$ is the tolerance and  $f_p^*$ is the optimal
	value of
	problem $p$ given in the literature,  we say
	$s\in{\mathcal{S}}$  solves $p\in{\mathcal{P}}$ up to the
	convergence test.
	Let $t_{p,s}$ be the least  number of function evaluations required by
	the solver  $s$ to solve the problem $p$
	up to the convergence test. If  $s$ does not satisfy the
	convergence test for  $p$ within the maximal number of function
	evaluations, then let $ t_{p,s}=\infty$.
	
%	Define the performance ratio
%	\begin{align*}
%		r_{p,s}\stackrel{\mathrm{def}}{=}\frac{t_{p,s}}
%		{\min\{t_{p,s}:s\in\mathcal{S}\}}.
%	\end{align*}

	The performance profile of $s$ is defined as
	\begin{align*}\label{ra}
		\pi_{s}(\alpha)=\frac{1}{n_{p}}
		\text{size}\big\{p\in\mathcal{P}:r_{p,s}\leq\alpha\big\},
	\end{align*}
	where $	r_{p,s}=\frac{t_{p,s}}
			{\min\{t_{p,s}:s\in\mathcal{S}\}}$ is the performance ratio,
			$n_p$ is the number of test problems in $\mathcal{P}$.
%			 and $\#$
%	indicates the cardinal number of a set.
	%From \cite{Dolan2002},
 		Hence,	$\pi_{s}(1)$ is the proportion of problems that 	$s$
		 performs better than  other solvers in $\mathcal{S}$   and
	$\pi_{s}(\alpha)$ is the proportion of problems  solved by  $s$ with a
	performance ratio at most
	$\alpha$.
	The better solver is indicated by the higher $\pi_{s}(\alpha)$. %curve.
%	The higher the value of
%	$\pi_{s}(\alpha)$, the better the performance of $s$.
	
	%The problem set has shown in Table \ref{table5} for a total of 27 test
%	problems.
	%We take $\tau=10^{-3}$ and $10^{-5}$, respectively. The performance profile
%	is
%	presented in Figure \ref{fcoby}. Figure \ref{f1} shows that for
%	$\tau=10^{-3}$,
%	DFLM-FD has the most wins when $\log_2(\alpha)<0.4$, but   it could
%	only solve 	$89\%$ of the problems.  DFLM-OSSv1 and DFLM-OSSv2 seem  much more
%	competitive
%	%as displayed by the height of its performance profile
% for 	$\log_2(\alpha)>0.4	$
%	and   could solve about $94\%$ of the problems.
%	For $\tau=10^{-5}$, as shown in Figure \ref{f2}, DFLM-OSSv1
%	and
%	DFLM-FD solve
%	43\% of the test problems   better than
%	the other solvers, DFLM-OSSv2 solves  38\% of the problems better than
%	the other.  DFLM-OSSv1 finally solves 90.7\% of
%	the  problems, while DFLM-OSSv2 solves 88\% and DFLM-FD solves 75.5\%.
%	

We take $\tau=10^{-3}$ and $10^{-5}$, respectively.
%A log scale of the performance profiles is presented
%and plotted:
%	\begin{align*}
%		\pi_{s}(\alpha)=\frac{1}{n_{p}}
%		\#\big\{p\in\mathcal{P}:\log_2(r_{p,s})\leq\log_2(\alpha)\big\}.
%	\end{align*}
	 Figure \ref{f1} shows that for
	$\tau=10^{-3}$,
	DFLM-OSSv1 and
	DFLM-OSSv2 seem  much more 	competitive  for 	$\log_2(\alpha)>0.35$ and
	finally solve about $94\%$  of the problems, while DFLM-FD has the most
	wins when $\log_2(\alpha)<0.35$ but solves   fewer problems.
	
	For $\tau=10^{-5}$, as shown in Figure \ref{f2}, both
DFLM-OSSv1
	and 	DFLM-FD perform better on about 43\% of the problems, while
	 the percentages for DFLM-OSSv2 is about 38\%.
	 DFLM-OSSv1, DFLM-OSSv2 and DFLM-FD
	solve  about 90.7\%, 88\% and 75.5\%
		of the problems, respectively.
	
%	NEWUOA uses the least number of function evaluations on about 50 \% of the
%prob.lems, while the percentages for BFGS and CG are about 30 \% and 25 \%
%respectively
	
%	, DFLM-OSSv2 solves  38\% of the problems better than
%	the other.
%	DFLM-OSSv1 finally solves 90.7\% of
%	the  problems, while DFLM-OSSv2 solves 88\% and DFLM-FD solves 75.5\%.

	%wins according to the performance
	%profile.
	
	%It shows that DFRCNLS-FD-SR1 solves 32\% of the test problems
	%better
	%than
	%the other solvers and finally able to solve all
	%test problems,  DFRCNLS-OSSv1-SR1 solves 28\% of the  problems with
	%better
	%than the other  and finally  solves 92\% of the test problems,
	%DFRCNLS-OSSv2-SR1 solves 36\% of the  problems with   better
	%than the other  and finally solves 92\% of the  problems, while
	%COBYLA  solves 28\% of the  problems better than the other three  and
	%finally only able to solve 76\% of the  problems.
	\begin{figure}[htbp]
		\centering
		\begin{subfigure}{0.49\linewidth}
			\centering
			\includegraphics[width=1\linewidth]
			%		{LMOSS.perf_3.backtrack.png}
			{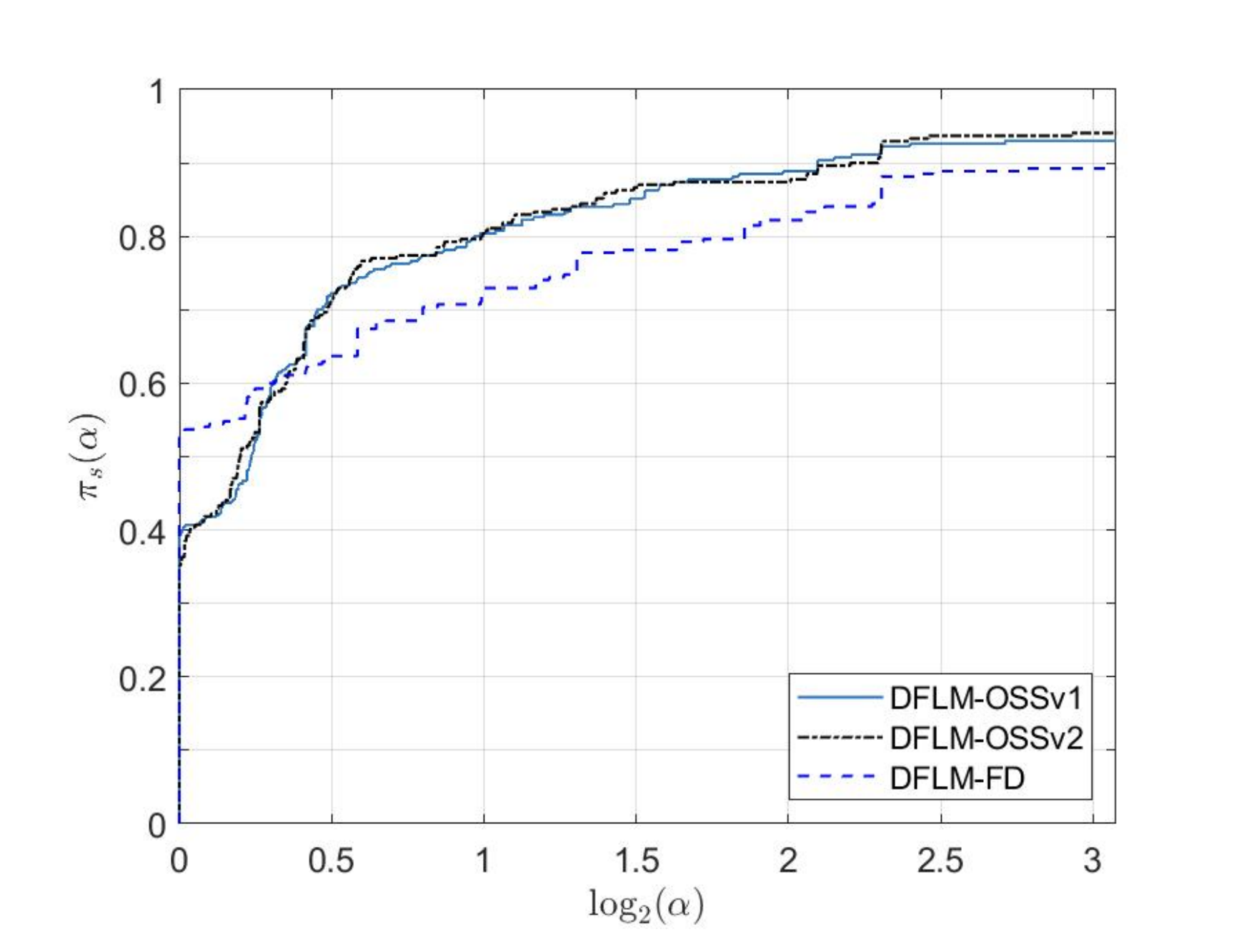}
			\caption{ $\tau=10^{-3}$ }
			\label{f1}%文中引用该图片代号
		\end{subfigure}
		\centering
		\begin{subfigure}{0.49\linewidth}
			\centering
			\includegraphics[width=1\linewidth]
			%		{LMOSS.perf_5.backtrack.png}
			{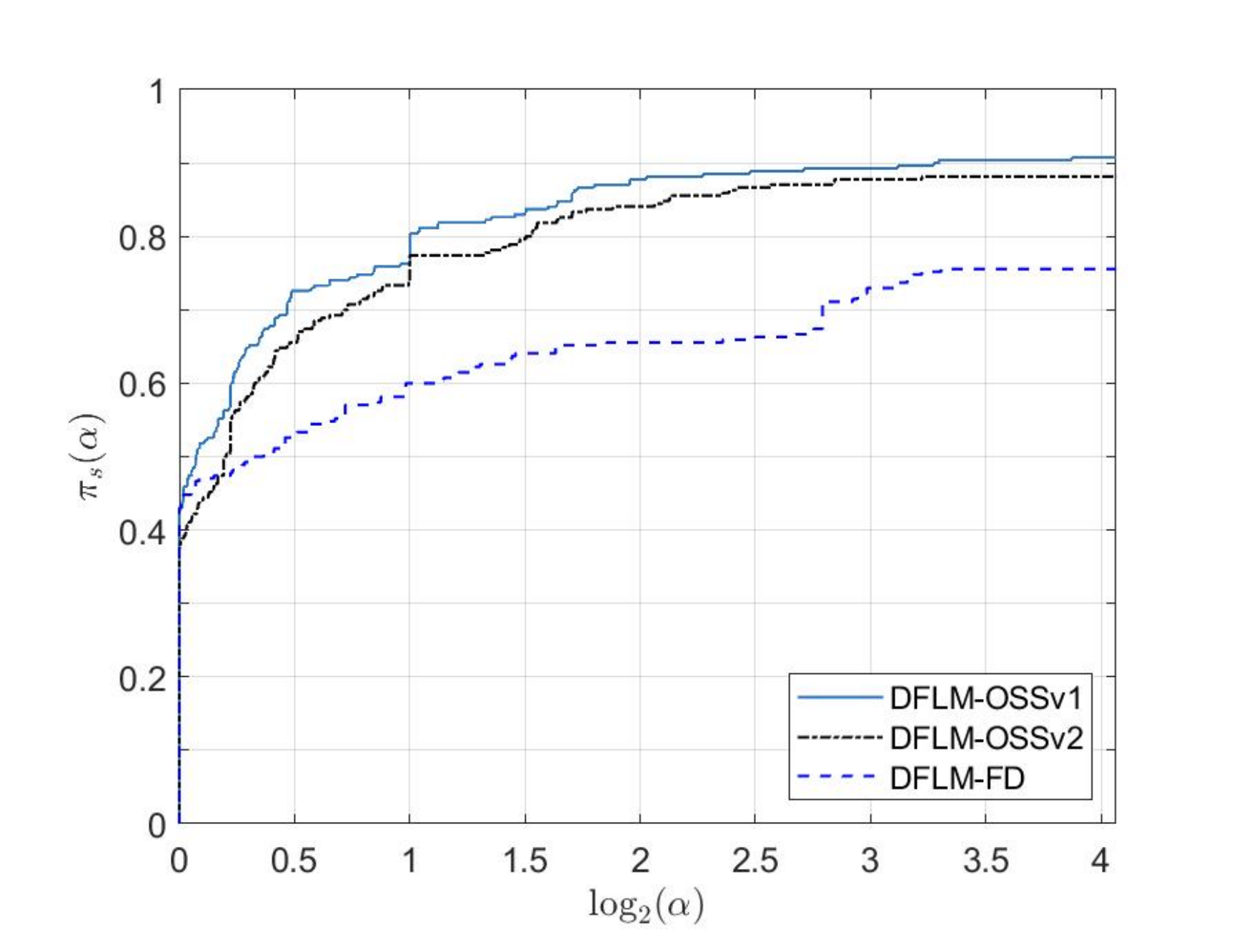}
			\caption{  $\tau=10^{-5}$ }
			\label{f2}%文中引用该图片代号
		\end{subfigure}
		\caption{Performance profile }
		\label{fcoby}
	\end{figure}
	
\textbf{Data Availability} Enquiries about data availability should be directed to the authors.

\textbf{Declarations} 

 \textbf{Conflict of interest} The authors have not disclosed any competing interests.

%It can be seen from the tables  and figures  that %DFLM-OSS
%%usually performs better than
%%DFLM-FD not only in number of function calculations but also in number of
%%iterations. ??
%%
%%It seems that
%for nonlinear least squares problems whose Jacobian matrices are unavailable
%or expensive to compute, the orthogonal spherical smoothing is a desirable
%technique to
%construct approximate Jacobian matrices, and the proposed derivative-free LM
%algorithm seems efficient.
%		
%		
		
%		\bibliographystyle{siam}
%	\bibliography{ref}	

	\end{document}